\documentclass[11pt]{amsart}

\usepackage{amsthm,amsmath,amsfonts,amssymb,mathrsfs,geometry,hyperref}
\usepackage{color}
\newtheorem{thm}{Theorem}
\newtheorem{defi}{Definition}

\newtheorem{lem}{Lemma}
\newtheorem{rmk}{Remark}
\def\torus{\mathbb{T}}
\def\real{\mathbb{R}}
\def\complex{\mathbb{C}}
\def\nat{\mathbb{N}}
\def\integer{\mathbb{Z}}
\def\D{\mathcal{D}}
\def\E{\mathcal{E}}

\title[Almost-periodic KAM equilibria]{KAM Theory for almost-periodic equilibria in one dimensional almost-periodic media}

\author[Y. An]{Yujia An}
\address{School of Mathematical Sciences, Beijing Normal University,
No. 19, XinJieKouWai St., HaiDian District, Beijing 100875, P. R. China}
\email{yjan@mail.bnu.edu.cn}

\author[R. de la Llave]{Rafael de la Llave}
\address{School of Mathematics\\
Georgia Institute of Technology \\
686 Cherry St. \\
Atlanta GA 30332, USA}
\email{rafael.delallave@math.gatech.edu}

\author[X. Su]{Xifeng Su}
\address{School of Mathematical Sciences, Laboratory of Mathematics and Complex Systems (Ministry of Education)\\
Beijing Normal University,
No. 19, XinJieKouWai St., HaiDian District, Beijing 100875, P. R. China}
\email{xfsu@bnu.edu.cn, billy3492@gmail.com}

\author[D. Wang]{Donghua Wang}
\address{School of Mathematical Sciences, Beijing Normal University,
No. 19, XinJieKouWai St., HaiDian District, Beijing 100875, P. R. China}
\email{wangdonghua@mail.bnu.edu.cn}

\author[D. Yao]{Dongyu Yao}
\address{School of Mathematical Sciences, Beijing Normal University,
No. 19, XinJieKouWai St., HaiDian District, Beijing 100875, P. R. China}
\email{yaodongyu@mail.bnu.edu.cn}

\date{\today}

\begin{document}
\maketitle


\begin{abstract}
We consider one dimensional chains of interacting particles subjected to one dimensional almost-periodic media.
We formulate and prove two KAM type theorems corresponding to both short-range and long-range interactions respectively. 
Both theorems presented have an a posteriori format and establish the existence of almost-periodic equilibria. The new part here is that the potential function is given by some almost-periodic function with infinitely many incommensurate frequencies. 

In both cases, we do not need to assume that the system is close to integrable. We will show that if there exists an approximate solution for the functional equations, which satisfies some appropriate non-degeneracy conditions, then a true solution nearby is obtained. This procedure may  be used to validate efficient numerical computations.

Moreover, to well understand the role of almost-periodic media which can be approximated by quasi-periodic ones, we present a different approach -- the step by step increase of complexity method -- to the study of the above results of the almost-periodic models.
\end{abstract}


\section{Introduction}
We consider the models in statistical mechanics \cite{Rue69} to describe one dimensional chains of interacting particles over a substratum. We assume that 
\begin{itemize}
\item the state of each site is given by a real variable;
\item the interactions are translation invariant;
\item the dependence on variables of the system is almost-periodic.
\end{itemize}
We could allow many-body interactions or even infinite range interactions where the interactions among far away particles decrease with the distances. These long-range interactions models admit a clear physical interpretation.

More precisely, the state of a system is given by $u=\{u_n\}_{n\in\mathbb Z}$ with $u_n\in\mathbb{R}$, and we call it  the configuration of that system. The physical states are then selected to be the critical points of a formal energy functional:
\begin{equation}\label{long}
\mathscr S(u)=\sum_{i\in\mathbb Z}\sum_{L=0}^{\infty}\widehat H_L(u_i\alpha, u_{i+1}\alpha,\cdots,u_{i+L}\alpha),
\end{equation}
where $\widehat H_L: \big(\mathbb T^{\mathbb N}\big)^{L+1}\to\mathbb R$ and $\alpha\in[0,1]^{\mathbb N}$ is rationally independent, that is, for any finite segments of $\alpha$ are rationally independent.

Note that the expression of the energy in \eqref{long} includes a sum over all $L$, which means each particle interacts with all the other particles.  The form of $\widehat H_L$ encodes the almost-periodic properties of the media and the form of the energy functional \eqref{long} encodes the fact that the energy is independent of shifting the index of the particles.

An interesting particular case of \eqref{long} is that each particle only interacts with the nearest neighbor particles.
These short-range interaction models have the same spirit as the famous Frenkel-Kontorova models \cite{FK39} for the periodic potential. That is, they could be obtained as an approximation of the long-range interaction models by arguing heuristically that the nearest terms are the most important. 

Indeed, let 
\[
\widehat{H}_0(u_n \alpha) = - \widehat{V}(u_n \alpha), \quad  \widehat{H}_1(u_n \alpha, u_{n+1}\alpha ) = \frac{1}{2} (u_n-u_{n+1})^2,\quad \widehat{H}_L=0\quad \forall L\geq2,
\]
and \eqref{long} becomes 
\begin{equation}\label{short}
\mathscr S(\{u_n\}_{n\in\mathbb Z})=\sum_{n\in\mathbb Z}[\frac12(u_n-u_{n+1})^2-V(u_n)],
\end{equation}
where $V$ is an almost-periodic function $V(\theta)=\widehat{V}(\theta\alpha)$, $\widehat V:\mathbb T^{\mathbb N}\to\mathbb R$, with infinite-dimensional rationally independent frequency $\alpha\in[0,1]^{\mathbb N}$.

Our goal in this paper is to show the existence of almost-periodic equilibrium solutions of both \eqref{long} and \eqref{short} under the assumption that the system is real analytic. 

As in \cite{ref2, SDlL12}, we are interested in finding ``line-like" configurations of the form 
 \[
 u_n=h(n\omega),\ \forall n\in\mathbb Z
 \]
 where $\omega\in\mathbb R$ is called the rotation number, $h(\theta)=\theta+\tilde h(\theta)$ and $\tilde h$ is an almost-periodic function with the basic frequency $\alpha\in[0,1]^{\mathbb N}$. The function $h$ is often referred as the ``hull'' function of the configuration in solid state physics.
 
Following \cite{Moser'66a, Moser'66b, Zeh75, ref1, ref6, ref7, ref2, SDlL12} and many others, we will obtain our main results, Theorems \ref{sr} and \ref{lr}, which are KAM type and will be presented in an a posteriori format. In fact, we  will prove that if there is an approximate solution of the equilibrium equations, which satisfies some non-degeneracy condition, then, we can find an exact solution near the approximate solution. The method of proof is based on the quasi-Newton method that overcomes the small divisors generated by the frequencies.

Different from previous works, we will encounter at the same time two difficulties below:
\begin{itemize}
\item the fact that the interaction is almost-periodic;
\item the fact that the interaction is long-range.
\end{itemize}

Note that the long-range difficulty is overcome both for the periodic interactions in \cite{ref6, ref7} and for the quasi-periodic interactions in \cite{SDlL12}. 

In this paper, we would mainly follow the idea of \cite{ref3} (see also \cite{MR4091501}), that is,
we first introduce the notion of infinite dimensional Diophantine vectors (see Section~\ref{sec: diophantine}) and we will revisit the KAM scheme to obtain the almost-periodic response equilibria with less analytic radius in Sections~\ref{S4} and \ref{S5} for short-range and  long-range cases respectively. See e.g. \cite{FontichLS, MR4026994, MR4227050, MR4227051, MR4666300, MR4710790} and so on for other recent works in almost-periodic KAM theory.

As for the almost-periodicity, one could also think of the definition of an almost-periodic function that is the limit of quasi-periodic ones with an increasing number of frequencies and obtain non-trivial almost-periodic solutions as in \cite{Poschel02} approximating by quasi-periodic solutions with an increasing number of frequencies. We do go this direction in Section~\ref{sec: finite approximation} and prove Theorem~\ref{new}, and we think it would be more efficient for numerical explorations of finding these almost-periodic equilibria.

In other directions, for instance, in \cite{Aubry90, Trevino19}, under the assumptions that the potential function $V$ is sufficiently enough, the authors use the anti-integrable limits approach to obtain multiple equilibria; \cite{GGP06}
uses topological methods to establish the existence of the minimal configurations for one-dimensional Fibonacci quasi-crystals, while \cite{GPT17} applies the ergodic optimization method to search minimal configurations.

\section{Preliminary}
\subsection{Function spaces}

Firstly, we will study the analytic function on the thickened infinite dimensional torus.
For $\rho>0$, we define the thickened infinite dimensional torus $\mathbb{T}^{\mathbb{N}}_\rho:=\{(\theta_j)_{j\in\mathbb{N}}\in\mathbb C^{\mathbb N}:{\rm Re}\theta_j\in\mathbb{T},|{\rm Im}\theta_j|\le\rho\langle\langle j\rangle\rangle^s,\forall j\in\mathbb{N}\}$, where $\mathbb{T}=\mathbb{R}/(2\pi\mathbb{Z})$, $\langle\langle j\rangle\rangle=\max\{|j|,1\}$, $s>0$. 

\begin{defi}
We denote the set of infinite integer vectors with finite support, $$\mathbb Z_*^{\mathbb N}:=\{k\in\mathbb Z^{\mathbb N}:|k|_s:=\sum_{j\in\mathbb{N}}\langle\langle j\rangle\rangle^s|k_j|<+\infty\}.$$ 
\end{defi}

\begin{defi}
    Given a Banach space $(X,\|\cdot\|_X)$, we define the space of analytic functions $\mathbb T_{\rho}^{\mathbb N}\to X$ as $$\mathcal H(\mathbb T_{\rho}^{\mathbb N},X):=\{ h(\sigma)=\sum_{k\in\mathbb Z_*^{\mathbb N}} h_ke^{ik \cdot \sigma}\in\mathcal F: \| h \|_{\rho}:=\sum_{k\in\mathbb Z_*^{\mathbb N}}\|h_k\|_Xe^{\rho|k|_s} < \infty \},$$ 
    where $\mathcal{F}$ denotes the space of pointwise absolutely convergent formal Fourier series
    $\mathbb T_{\rho}^{\mathbb N}\to X$, $$h(\sigma)=\sum_{k\in\mathbb Z_*^{\mathbb N}} h_ke^{ik \cdot \sigma},\quad h_k\in X.$$
    We denote $\mathscr A_{\rho}:=\mathcal{H}(\mathbb T_{\rho}^{\mathbb N},\mathbb C)$.
\end{defi}

For the purpose of discussing the properties of analytic functions, the following definitions are necessary.

\begin{defi}
    Given Banach spaces $X$ and $Y$, we define the space $\mathcal{M}_k(X,Y)$ of the $k$-linear and continuous forms endowed by $$\|M\|_{\mathcal{M}_k(X,Y)}:=\sup_{\|u_1\|_X,\cdots,\|u_k\|_X\le1}\|M(u_1,\cdots,u_k)\|_Y, \quad\forall M\in\mathcal{M}_k(X,Y).$$
\end{defi}

\begin{defi}
    We denote $\mathscr A_{\rho}^r$ as the subspace of $\mathscr A_{\rho}$, 
    $$\mathscr A_{\rho}^r:=\{ h\in\mathscr A_{\rho}: D^r h\in \mathcal{H}(\mathbb T_{\rho}^{\mathbb N},\mathcal M_r(l^{\infty}_{\mathbb C},\mathbb C)), \|h\|_{\mathscr A_{\rho}^r}:=\|D^r h\|_{\rho}+|\langle h \rangle|\},$$
    where $l^{\infty}_{\mathbb C}:=\{(a_\lambda)_{\lambda\in\mathbb N}\in\mathbb C^{\mathbb N}:\sup_i|a_i|<+\infty\}$, and $\langle h\rangle$ denotes the average of $h$, 
    $$\langle h \rangle:=\int_{\mathbb T^{\mathbb N}}h(\sigma) d\sigma:=\lim_{N\to +\infty} \frac{1}{(2\pi)^{N}}\int_{\mathbb T^N} h(\sigma) d\sigma_1\cdots d\sigma_N. $$
    
    And we denote by $\overset{\circ}{\mathscr A_\rho}$ the subset of $\mathscr A_\rho$ with zero average. 
\end{defi}

\begin{rmk}
    The existence of $\langle h\rangle$ is provided by Lemma 2.4 in \cite{ref5}. According to the definition of $\rho$-norm and Cauchy estimate, if $\iota>0$, we have $\mathscr A_{\rho}^{r-1}\subset\mathscr A_{\rho+\iota}^{r-1}\subset\mathscr A_{\rho}^r$.
\end{rmk}

Now, we can define the real analytic almost periodic function $f:\mathbb R\to\mathbb R$.

\begin{defi}\label{rational independent}
$\alpha=(\alpha_{\lambda})_{\lambda\in\mathbb N}\in\mathbb R^{\mathbb N}$ is said to be rationally independent if for any finite segments $(\alpha_{\lambda_1},\cdots,\alpha_{\lambda_n})$, we have $\sum_{i=1}^n\alpha_{\lambda_i}k_i\ne0$, for all $ (k_1,\cdots,k_n)\in\mathbb Z^n\setminus\{0\} $.  
\end{defi}

\begin{defi}
    The function $f:\mathbb R\to\mathbb R$ is said to be a real analytic almost periodic function in the strip $\rho$ with basic frequency $\alpha\in\mathbb R^{\mathbb N}$ and shell function $F: \mathbb T^{\mathbb N} \to \mathbb R$, if $f(\theta)=F(\alpha\theta)$, $\alpha$ is rational independent, and $F$ can be extended to an analytic function in $\mathscr A_{\rho}$. And we define the $\rho$-norm of $f$ as $\|f\|_{\rho}:=\|F\|_{\rho}$. We denote by $AP(\alpha)$ the set of real analytic almost periodic functions with the frequency $\alpha$.
\end{defi}

In order to evaluate the interactions between the particles, we need to introduce the analytic functions on $(\mathbb T_{\rho}^{\mathbb N})^{L+1}$, $L\in\mathbb N$.
\begin{defi}
    Given a Banach space $(X,\|\cdot\|_X)$, we define the space of analytic functions $({\mathbb T_{\rho}^{\mathbb N}})^{L+1}\to X$ as 

    \begin{align*}
        \mathcal{H}_{L}((\mathbb T_{\rho}^{\mathbb N})^{L+1},X):=&\{ H(\sigma_0,\sigma_1,\cdots,\sigma_L)=\sum_{(k_0,\cdots,k_L)\in(\mathbb Z_*^{\mathbb N})^{L+1}} H(k_0,\cdots,k_L) e^{i\sum_{j=0}^Lk_j\cdot\sigma_j}\in\mathcal{F}_L:\\
        &\|H\|_{L,\rho}:=\sum_{(k_0,\cdots,k_L)\in(\mathbb Z_*^{\mathbb N})^{L+1}}\|H(k_0,\cdots,k_L)\|_X e^{\sum_{j=0}^L\rho|k_j|_s}\},
    \end{align*}

    where $\mathcal F_L$ denotes the space of pointwise absolutely convergent formal Fourier series $(\mathbb T_{\rho}^{\mathbb N})^{L+1}\to X$, 
    $$H(\sigma_0,\sigma_1,\cdots,\sigma_L)=\sum_{(k_0,\cdots,k_L)\in({\mathbb Z_*^{\mathbb N}})^{L+1}} H(k_0,\cdots,k_L) e^{i\sum_{j=0}^Lk_j\cdot\sigma_j},\quad H(k_0,\cdots,k_L)\in X.$$
    We denote $\mathcal{H}_{L,\rho}:=\mathcal H( (\mathbb T_{\rho}^{\mathbb N})^{L+1} ,\mathbb C)$.
    
\end{defi}

\begin{defi}
    We denote $\mathcal{H}_{L,\rho}^r$ as the subspace of $\mathcal{H}_{L,\rho}$, $$\mathcal{H}_{L,\rho}^r:=\{ h\in\mathcal{H}_{L,\rho}:D^i H\in\mathcal{H}( (\mathbb T_{\rho}^{\mathbb N})^{L+1} , \mathcal{M}_r((l_{\mathbb C}^{\infty})^{L+1}, \mathbb C) ) ,i\le r\},$$
    where $(l_{\mathbb C}^{\infty})^{L+1}:=\{ (a_0,\cdots,a_L): a_i\in l_{\mathbb C}^{\infty},\,i=0,\cdots,L \}.$
\end{defi}

\subsection{Diophantine properties}\label{sec: diophantine}
The Diophantine vector is $\alpha\in[0,1]^{\mathbb N}$ such that
\begin{equation}\label{dio1}
|\alpha\cdot k|\ge\frac{\nu_0}{\prod_{j\in\mathbb N}(1+\langle\langle j\rangle\rangle^{1+\tau_0}|k_j|^{1+\tau_0})}, \quad \forall k\in\mathbb Z^{\mathbb N}\setminus\{0\}.
\end{equation} 
We are interested in the numbers  $\omega\in\mathbb R$  such that
\begin{equation}\label{dio2}
|\omega\alpha\cdot k-2n\pi|\ge\frac{\nu}{\prod_{j\in\mathbb N}(1+\langle\langle j\rangle\rangle^{1+\tau}|k_j|^{1+\tau})}, \quad\forall k\in\mathbb Z^{\mathbb N}\setminus\{0\},\ \forall n\in\mathbb Z,
\end{equation}
where $\alpha$ is a Diophantine vector, $\nu_0, \tau_0, \nu,\tau$ are positive numbers.

We denote by $\overline{\mathscr D}(\nu_0, \tau_0)$ the set of $\alpha\in [0,1]^{\mathbb N}$ satisfying \eqref{dio1}.
We denote by $\mathscr D(\nu,\tau;\alpha)$ the set of $\omega$ satisfying \eqref{dio2}  and $\mathscr D(\tau;\alpha)=\cup_{\nu>0}\mathscr D(\nu,\tau;\alpha)$. 

Actually, most $\alpha\in[0,1]^{\mathbb N}$ satisfy (\ref{dio1}) according to \cite{ref3,MR4091501}. More precisely,
\begin{lem}\label{l2}
For $\nu\in(0,1)$, $\tau>1$, there exists $C(\tau)>0$ such that $$m([0,1]^\mathbb N\setminus \overline{\mathscr D}(\nu,\tau))\le C(\tau)\nu,$$
where $m(\cdot)$ is the product probability measure.
\end{lem}

We only need to prove that there are sufficiently many $\omega$ satisfying our Diophantine condition \eqref{dio2} for given $\alpha$ satisfying (\ref{dio1}). The below measure estimate means that $\mathscr D(\tau;\alpha)$ is of full Lebesgue measure. Hence our discussion is meaningful. 
\begin{lem}[Measure estimates]
If $\alpha\in [0,1]^{\mathbb N}$ satisfies (\ref{dio1}) and $\tau>\tau_0+1$, then $\mathscr D(\tau;\alpha)$ is of full Lebesgue measure. 
\end{lem}
\begin{proof}
Let $A>0$. Consider the sets 
\begin{equation*}
\begin{aligned}
\mathscr B_{k,n}&=\{\omega:|\omega\alpha\cdot k-2n\pi|<\frac{\nu}{\prod_{j\in\mathbb N}(1+\langle\langle j\rangle\rangle^{1+\tau}|k_j|^{1+\tau})}\}\\
&=\{\omega:|\omega-\frac{2n\pi}{\alpha\cdot k}|<\frac{\nu}{|\alpha\cdot k|\prod_{j\in\mathbb N}(1+\langle\langle j\rangle\rangle^{1+\tau}|k_j|^{1+\tau})}\}
\end{aligned}
\end{equation*}
and$$\mathscr B_{k,n,A}=\mathscr B_{k,n}\cap [A,1.01A].$$
Obviously, $$[A,1.01A]\setminus\mathscr D(\nu,\tau;\alpha)=\cup_{k\in\mathbb Z^{\mathbb N}\setminus\{0\},n\in\mathbb N}\mathscr B_{k,n,A}.$$
Therefore, we only need to show $$|\cup_{k,n}\mathscr B_{k,n,A}|<\nu C(\tau,\tau_0,\nu_0,A).$$ 
We note $$|\mathscr B_{k,n}|=\frac{2\nu}{|\alpha\cdot k|\prod_{j\in\mathbb N}(1+\langle\langle j\rangle\rangle^{1+\tau}|k_j|^{1+\tau})}.$$
If $\mathscr B_{k,n,A}\ne\varnothing$, we get $A<\frac{2n\pi}{\alpha\cdot k}+\frac{\nu}{|\alpha\cdot k|\prod_{j\in\mathbb N}(1+\langle\langle j\rangle\rangle^{1+\tau}|k_j|^{1+\tau})}\le1.01A$ or $A\le\frac{2n\pi}{\alpha\cdot k}-\frac{\nu}{|\alpha\cdot k|\prod_{j\in\mathbb N}(1+\langle\langle j\rangle\rangle^{1+\tau}|k_j|^{1+\tau})}<1.01A$.
Then, $$\sharp\{n:\mathscr B_{k,n,A}\ne\varnothing\}\le0.02A\frac{|\alpha\cdot k|}{2\pi}+2.$$
Therefore, 
\begin{equation*}
\begin{split}
&\quad|\cup_{k,n}\mathscr B_{k,n,A}|\\
&\le\sum_{k,n,\mathscr B_{k,n,A}\ne\varnothing}\frac{2\nu}{|\alpha\cdot k|\prod_{j\in\mathbb N}(1+\langle\langle j\rangle\rangle^{1+\tau}|k_j|^{1+\tau})}\\
&\le\sum_{k\in\mathbb Z^{\mathbb N}\setminus\{0\}}\frac{2\nu}{|\alpha\cdot k|\prod_{j\in\mathbb N}(1+\langle\langle j\rangle\rangle^{1+\tau}|k_j|^{1+\tau})}\cdot(0.02A\frac{|\alpha\cdot k|}{2\pi}+2)\\
&\le\frac{2\nu\cdot0.02A}{2\pi}\sum_{k\in\mathbb Z^{\mathbb N}\setminus\{0\}}\frac{1}{\prod_{j\in\mathbb N}(1+\langle\langle j\rangle\rangle^{1+\tau}|k_j|^{1+\tau})}+\frac{4\nu}{\nu_0}\sum_{k\in\mathbb Z^{\mathbb N}}\prod_{j\in\mathbb N}\frac{1+\langle\langle j\rangle\rangle^{1+\tau_0}|k_j|^{1+\tau_0}}{1+\langle\langle j\rangle\rangle^{1+\tau}|k_j|^{1+\tau}}\\
&\le\frac{2\nu\cdot0.02A}{2\pi}\prod_{j\in\mathbb N}(1+\sum_{\substack{k_j=-\infty\\k_j\ne0}}^{\infty}\frac{1}{1+\langle\langle j\rangle\rangle^{1+\tau}|k_j|^{1+\tau}})\\
&\quad+\frac{4\nu}{\nu_0}\prod_{j\in\mathbb N}(1+\sum_{\substack{k_j=-\infty\\k_j\ne0}}^{+\infty}\frac{1+\langle\langle j\rangle\rangle^{1+\tau_0}|k_j|^{1+\tau_0}}{1+\langle\langle j\rangle\rangle^{1+\tau}|k_j|^{1+\tau}})\\
&\le\frac{2\nu\cdot0.02A}{2\pi}\prod_{j\in\mathbb N}(1+\frac{C_1(\tau)}{\langle\langle j\rangle\rangle^{1+\tau}})+\frac{4\nu}{\nu_0}\prod_{j\in\mathbb N}(1+\frac{C_2(\tau)}{\langle\langle j\rangle\rangle^{1+\tau}}+\frac{C_3(\tau,\tau_0)}{\langle\langle j\rangle\rangle^{\tau-\tau_0}})\\
&\le\frac{2\nu\cdot0.02A}{2\pi}C_4(\tau)+\frac{4\nu}{\nu_0}C_5(\tau,\tau_0)=\nu C(\tau,\tau_0,\nu_0,A).\qedhere 
\end{split}
\end{equation*}
\end{proof}

\subsection{Homological equations}
\begin{lem}\label{l1}
Let $\omega\in\mathbb{R},\alpha\in[0,1]^\mathbb{N}$, $|\omega\alpha\cdot k-2n\pi|\ge\frac\nu{\prod_{j\in\mathbb N}(1+\langle\langle j\rangle\rangle^{1+\tau}|k_j|^{1+\tau})}$ holds for $\forall k\in\mathbb Z^{\mathbb N}_*\setminus\{0\}$, $\forall n\in\mathbb Z$, where $\tau,\nu>0$. If $\hat\eta\in\overset{\circ}{\mathscr A_\rho}$, then equation 
\begin{equation}\label{homeq}
\hat\phi(\sigma+\omega\alpha)-\hat\phi(\sigma)=\hat\eta(\sigma)
\end{equation}
has the unique solution $\hat\phi$ with zero average, which satisfies
\begin{equation*}
\|\hat\phi\|_{\rho'}\le C\nu^{-1}{\rm exp}\left(\frac{\tilde\tau}{(\rho-\rho')^\frac1s}{\rm ln}\frac{\tilde\tau}{\rho-\rho'}\right)\|\hat\eta\|_\rho,
\end{equation*}
where $\tilde\tau=\tilde\tau(\tau,s)$.
\end{lem}
\begin{proof}
Substitute $\hat\phi(\sigma)=\sum_{k\in\mathbb Z^\mathbb{N}_*}\hat\phi_ke^{ik\cdot\sigma}$, $\hat\eta(\sigma)=\sum_{k\in\mathbb Z^\mathbb{N}_*}\hat\eta_ke^{ik\cdot\sigma}$ into the homological equation (\ref{homeq}),we obtain that
$$\hat\phi_k=\frac{\hat\eta_k}{e^{ik\cdot\omega\alpha}-1},$$
for all $k\in\mathbb Z^\mathbb{N}_*\setminus\{0\}$.

Note that there exists $C>0$, such that for all $t\in\mathbb R$, we have $C|e^{it}-1|\ge{\rm dist}(t,2\pi\mathbb Z)$ (taking $C=\frac{\pi}{2}$ is enough). Therefore,
\begin{equation*}
\begin{split}
|\hat\phi_k|\le&C{\rm dist}(k\cdot\omega\alpha,2\pi\mathbb Z)^{-1}|\hat\eta_k|\\
\le&C\nu^{-1}\prod_{j\in\mathbb N}(1+\langle\langle j\rangle\rangle^{1+\tau}|k_j|^{1+\tau})|\hat\eta_k|.
\end{split}
\end{equation*}
When $\hat\phi_0=0$, i.e. $\int_{\mathbb T_\rho^\mathbb{N}}\hat\phi(\sigma)d\sigma=0$, 
\begin{equation*}
\begin{split}
\|\hat\phi\|_{\rho'}=&\sum_{k\in\mathbb Z^\mathbb{N}_*}e^{\rho'|k|_s}|\hat\phi_k|=\sum_{k\in\mathbb Z^\mathbb{N}_*\setminus\{0\}}e^{\rho'|k|_s}|\hat\phi_k|\\
\le&\sum_{k\in\mathbb Z^\mathbb{N}_*\setminus\{0\}}e^{\rho'|k|_s}C\nu^{-1}\prod_{j\in\mathbb N}(1+\langle\langle j\rangle\rangle^{1+\tau}|k_j|^{1+\tau})|\hat\eta_k|\\
=&\sum_{k\in\mathbb Z^\mathbb{N}_*\setminus\{0\}}e^{\rho|k|_s}|\hat\eta_k|e^{-(\rho-\rho')|k|_s}C\nu^{-1}\prod_{j\in\mathbb N}(1+\langle\langle j\rangle\rangle^{1+\tau}|k_j|^{1+\tau})\\
\le&\mathop{\rm sup}\limits_{k\in\mathbb Z^\mathbb{N}_*\setminus\{0\}}e^{-(\rho-\rho')|k|_s}\prod_{j\in\mathbb N}(1+\langle\langle j\rangle\rangle^{1+\tau}|k_j|^{1+\tau})C\nu^{-1}\|\hat\eta\|_\rho\\
\le&C\nu^{-1}{\rm exp}\left(\frac{\tilde\tau}{(\rho-\rho')^{\frac1s}}{\rm ln\frac{\tilde\tau}{\rho-\rho'}}\right)\|\hat\eta\|_\rho,
\end{split}
\end{equation*}
where the final step is due to \cite{ref5} Lemma B.1, $\tilde\tau=\tilde\tau(\tau,s)$. By now we have finished the proof of the solution's existence, and then we are to prove the uniqueness. 

If there are two solutions with zero average, denoted by $\hat\phi,\hat\varphi$, then $$(\hat\phi-\hat\varphi)(\sigma+\omega\alpha)-(\hat\phi-\hat\varphi)(\sigma)=0,$$
then $$\|\hat\phi-\hat\varphi\|_{\rho'}=0,$$
then $\hat\phi=\hat\varphi$, which means that the solution is unique.
\end{proof}

\subsection{Weighted norm properties}
\begin{lem}\label{l3}
Let $\hat f,\hat g\in\mathscr A_\rho$, then $\|\hat f\cdot \hat g\|_\rho\le\|\hat f\|_\rho\|\hat g\|_\rho$.
\end{lem}
\begin{proof}
It's easy to check that $$\hat f\cdot \hat g(\sigma)=\sum_{k,l\in\mathbb Z^\mathbb{N}_*}\hat f_{k-l}\hat g_le^{ik\cdot\sigma},$$
then $$\|\hat f\cdot \hat g\|_\rho\le\sum_{k,l\in\mathbb Z^\mathbb{N}_*}|\hat f_{k-l}||\hat g_l|e^{\rho|k|_s}.$$
Using the triangular inequality of $s$-norm, we obtain that
$$|k|_s\le|k-l|_s+|l|_s,$$
then $$e^{\rho|k|_s}\le e^{\rho|k-l|_s}e^{\rho|l|_s},$$
therefore 
\begin{equation*}
   \|\hat f\cdot\hat g\|_\rho\le\sum_{k,l\in\mathbb Z^\mathbb{N}_*}|\hat f_{k-l}||\hat g_l|e^{\rho|k|_s}\le\|f\|_\rho\|g\|_\rho. \qedhere
\end{equation*}
\end{proof}

\begin{lem}\label{l0}
    Let $\hat h\in\mathscr A_{\rho}$, then for all $0<\delta\le\rho$, $\|\hat h\|_{\rho}\le\|\hat h\|_{\rho-\delta}^{\frac12}\cdot \|\hat h\|_{\rho+\delta}^{\frac{1}{2}}$.
\end{lem}
\begin{proof}
    By Cauchy-Schwarz inequality, we have:
    \begin{align*}
        \|\hat h\|_{\rho}&=\sum_{k\in\mathbb Z_*^{\mathbb N}}|\hat h_k|e^{\rho|k|_s}\\
        &\le(\sum_{k\in\mathbb Z_*^{\mathbb N}}|\hat h_k|e^{(\rho-\delta)|k|_s})^{\frac{1}{2}}\cdot (\sum_{k\in\mathbb Z_*^{\mathbb N}}|\hat h_k|e^{(\rho+\delta)|k|_s})^{\frac{1}{2}}=\|\hat h\|_{\rho-\delta}^{\frac12}\cdot \|\hat h\|_{\rho+\delta}^{\frac{1}{2}}. \qedhere
    \end{align*}    
\end{proof}

\begin{lem}[Cauchy estimates]\label{l4}
Let $\hat h(\sigma)=\sum\limits_{k\in\mathbb{Z}^\mathbb{N}_*}\hat h_k e^{ik\cdot\sigma}$, $\hat h\in\mathscr A_{\rho}^1$, $\alpha=(\alpha_\lambda)_{\lambda\in\mathbb{N}}\in[0,1]^\mathbb{N}$, for $\rho'<\rho$ we have that
$$
\|\partial_\alpha\hat h\|_{\rho'}\le(\rho-\rho')^{-1}\|\hat h\|_\rho.
$$
\end{lem}
\begin{proof}
Note that
$$
\partial_\alpha\hat h(\sigma)=\sum_{k\in\mathbb{Z}^\mathbb{N}_*}\hat h_ki(k\cdot\alpha) e^{ik\cdot\sigma}.
$$
Therefore,
\[
\begin{aligned}
\Vert \partial_\alpha\hat h\Vert_{\rho'}&=\sum_{k\in\mathbb{Z}^\mathbb{N}_*}|\hat h_k||k\cdot\alpha|e^{\rho'|k|_s}=\sum_{k\in\mathbb{Z}^\mathbb{N}_*}|\hat h_k||k\cdot\alpha|e^{\rho|k|_s}e^{-(\rho-\rho')|k|_s}\\
&\leq\sum_{k\in\mathbb{Z}^\mathbb{N}_*}|\hat h_k|\sum_{j\in\mathbb{N}}|k_j|e^{\rho|k|_s}e^{-(\rho-\rho')|k|_s}\leq\sum_{k\in\mathbb{Z}^\mathbb{N}_*}|\hat h_k||k|_se^{\rho|k|_s}e^{-(\rho-\rho')|k|_s}\\
&\leq\mathop{\rm sup}\limits_{k\in\mathbb Z^{\mathbb N}_*}|k|_se^{-(\rho-\rho')|k|_s}\Vert\hat h\Vert_\rho\leq(\rho-\rho')^{-1}\Vert\hat h\Vert_\rho,
\end{aligned}
\]
where the final step is due to the observation that $xe^{-ax}\le \frac1a$ holds for $\forall x\in\mathbb R$, when $a>0$.
\end{proof}

\begin{lem}\label{l5}
Let $\hat h\in\mathscr A_\rho$, $\widehat U\in\mathscr A_{\rho+\|\hat h\|_{\rho}+\iota}^2$, for some $\iota>0$. Define the operator $\Phi$ acting on analytic functions by $(\Phi[\hat h])(\sigma)=\widehat U(\sigma+\alpha\cdot\hat h(\sigma))$.

If $\|\hat h^*-\hat h\|_\rho \le\iota$, then $\Phi[\hat h^*]\in \mathscr A_\rho$.

If $\|\hat h^*-\hat h\|_\rho \le\iota$, $(D\Phi[\hat h]\widehat \Delta)(\sigma):=\partial_\alpha\widehat U(\sigma+\alpha\hat h(\sigma))\widehat \Delta(\sigma),$ then
\[
\|\Phi[\hat h^*]-\Phi[\hat h]-D\Phi[\hat h](\hat h^*-\hat h)\|_\rho\le \|\widehat U\|_{\mathscr A_{\rho+\|\hat h\|_{\rho}+\iota}^2}\cdot\|\hat h^*-\hat h\|_\rho^2.
\]
\end{lem}

\begin{proof}
By \cite[Lemma 2.15]{ref5}, $\Phi[\hat h^*]\in \mathscr A_\rho$. Using \cite[Lemma 2.15]{ref5} and Lemma \ref{l3}, we have
\begin{equation*}
\begin{split}
&\|\Phi[\hat h^*]-\Phi[\hat h]-D\Phi[\hat h](\hat h^*-\hat h)\|_\rho\\
=&\|\int_0^1\int_0^1\partial_{\alpha}\partial_{\alpha}\widehat U(\sigma+\alpha\hat h(\sigma)+st\alpha(\hat h^*-\hat h))t(\hat h^*-\hat h)^2dsdt\|_{\rho}\\
\le&\int_0^1\int_0^1\|\partial_{\alpha}\partial_{\alpha}\widehat U(\sigma+\alpha\hat h(\sigma)+st\alpha(\hat h^*-\hat h))\|_{\rho}dsdt\cdot\|\hat h^*-\hat h\|_{\rho}^2\\
\le&\int_0^1\int_0^1\|\partial_{\alpha}\partial_{\alpha}\widehat U(\sigma)\|_{\rho+\|\hat h\|_\rho+\iota}dsdt\cdot\|\hat h^*-\hat h\|_{\rho}^2\le \|\widehat U\|_{\mathscr A_{\rho+\|\hat h\|_{\rho}+\iota}^2}\cdot\|\hat h^*-\hat h\|_\rho^2. \qedhere 
\end{split} 
\end{equation*}
\end{proof}

\begin{lem}\label{l7}
Let $\hat h\in\mathscr A_{\rho}$, $\widehat H\in\mathcal{H}_{L,\rho+\|\hat h\|_{\rho}+\iota}$, for some $\iota>0$. Define the operator $\Psi$ acting on analytic functions by $(\Psi[\hat h])(\sigma):=\widehat H(\sigma+\hat h(\sigma)\alpha,\sigma+\omega\alpha+\hat h(\sigma+\omega\alpha),\cdots,\sigma+L\omega\alpha+\hat h(\sigma+L\omega\alpha)\alpha)$, where $\alpha\in [0,1]^{\mathbb N}$, $\omega\in\mathbb R$.

If $\|\hat h^*-\hat h\|_{\rho}\le\iota$, then $\Psi[\hat h^*]\in\mathscr A_{\rho}$.

If $\|\hat h^*-\hat h\|_{\rho}\le\iota$, then $\|\Psi[\hat h^*]-\Psi[\hat h]\|_{\rho}\le(L+1)\|D\widehat H\|_{L,\rho+\|\hat h\|_{\rho}+\iota}\|\hat h^*-\hat h\|_{\rho}$.

If $\|\hat h^*-\hat h\|_{\rho}\le\iota$, $(D\Psi[\hat h]\widehat\Delta)(\sigma):=\sum_{j=0}^{L}\partial_{\alpha}^{(j)}\widehat H(\sigma+\hat h(\sigma)\alpha,\cdots,\sigma+L\omega\alpha+\hat h(\sigma+L\omega\alpha)\alpha)(\widehat\Delta(\sigma+j\omega\alpha))$,  then \[
\|\Psi[\hat h^*]-\Psi[\hat h]-D\Psi[\hat h](\hat h^*-\hat h)\|_\rho\le (L+1)^2 \|D^2\widehat H\|_{L,\rho+\|\hat h\|_{\rho}+\iota}\|\hat h^*-\hat h\|_\rho^2.
\]

\end{lem}
\begin{proof}
    Actually, we only need to prove that if $\hat h\in\mathscr A_{\rho}$, $\widehat H\in\mathcal{H}_{L,\rho+\|\hat h\|_{\rho}}$, then $\Psi[\hat h]\in\mathscr A_{\rho}$. The remaining proofs are similar to Lemma \ref{l5}.

    Note that
    \begin{equation}\label{aa}
        \begin{aligned}
            \Psi[\hat h](\sigma)&=\widehat H(\sigma+\hat h(\sigma)\alpha,\cdots,\sigma+L\omega\alpha+\hat h(\sigma+L\omega\alpha)\alpha)\\
            &=\sum_{k_0,\cdots,k_L\in\mathbb Z_*^{\mathbb N}}\widehat H(k_0,\cdots,k_L)e^{i\sigma\cdot k_0}\cdots e^{i\sigma\cdot k_L}e^{i\hat h(\sigma)\alpha\cdot k_0}\cdots e^{i(L\omega+\hat h(\sigma+L\omega\alpha))\alpha\cdot k_L}.
        \end{aligned}
    \end{equation}
    For all $k_j\in\mathbb Z_*^{\mathbb N}$, we have
    \begin{align*}
        e^{i\hat h(\sigma+j\omega\alpha)\alpha\cdot k_j}&=\sum_{n_j\in\mathbb N}\frac{i^{n_j}}{n_j!}\hat h(\sigma+j\omega\alpha)^{n_j}(\alpha\cdot k_j)^{n_j}\\
        &=\sum_{n_j\in\mathbb N} \frac{i^{n_j}}{n_j!}(\alpha\cdot k_j)^{n_j}\sum_{k_{j,1},\cdots,k_{j,n_j}\in\mathbb Z_*^{\mathbb N}}\hat h_{k_{j,1},}\cdots \hat h_{k_{j,n_j}}e^{i(k_{j,1}+\cdots+k_{j,n_j})\cdot(\sigma+j\omega\alpha)},
    \end{align*}    
    where $\hat h(\sigma)=\sum_{k\in\mathbb Z_*^{\mathbb N}}\hat h_ke^{ik\cdot\sigma}$.
    Substituting into the (\ref{aa}), we obtain
    \begin{align*}
        \Psi[\hat h](\sigma)&=\sum_{k_0,\cdots,k_L\in\mathbb Z_*^{\mathbb N}}\widehat H(k_0,\cdots,k_L)e^{i\sigma\cdot k_0}\cdots e^{i\sigma\cdot k_L}e^{i\omega\alpha\cdot k_1}\cdots e^{iL\omega\alpha\cdot k_L}\cdot\\
        &\quad\cdot \prod_{j=0}^{L}(\sum_{n_j\in\mathbb N}\frac{i^{n_j}}{n_j!}(\alpha\cdot k_j)^{n_j}\sum_{k_{j,1},\cdots,k_{j,n_j}\in\mathbb Z_*^{\mathbb N}}\hat h_{k_{j,1}}\cdots \hat h_{k_{j,n_j}}e^{i(k_{j,1}+\cdots+k_{j,n_j})\cdot(\sigma+j\omega\alpha)})\\
        &=\sum_{n_0,\cdots,n_L\in\mathbb N}(\prod_{j=0}^L\frac{i^{n_j}}{n_j!})\sum_{\substack{k_0,k_{0,1},\cdots,k_{0,n_0}\in\mathbb Z_*^{\mathbb N}\\ \cdots\\k_L,k_{L,1},\cdots,k_{L,n_L}\in\mathbb Z_*^{\mathbb N}}}(\widehat H(k_0,\cdots,k_L)\prod_{j=0}^L (\alpha\cdot k_j)^{n_j}\cdot\\
        &\quad\cdot\hat h_{k_{j,1}}\cdots\hat h_{k_{j,n_j}}e^{i(k_j+k_{j,1}+\cdots+k_{j,n_j})\cdot j\omega\alpha}e^{i(k_j+k_{j,1}+\cdots+k_{j,n_j})\cdot \sigma}).
    \end{align*}
    Since $\Psi[\hat h](\sigma)=\sum_{u\in\mathbb Z_*^{\mathbb N}}\widehat{\Psi[\hat h]}_u e^{iu\cdot\sigma}$, we get 
    \begin{align*}
        \widehat{\Psi[\hat h]}_u&=\sum_{n_0,\cdots,n_L\in\mathbb N}(\prod_{j=0}^L\frac{i^{n_j}}{n_j!})\sum_{\sum_{t=0}^L k_t+k_{t,1}+\cdots+k_{t,n_t}=u}(\widehat H(k_0,\cdots,k_L)\prod_{j=0}^L (\alpha\cdot k_j)^{n_j}\cdot\\
        &\quad\cdot\hat h_{k_{j,1}}\cdots\hat h_{k_{j,n_j}}e^{i(k_j+k_{j,1}+\cdots+k_{j,n_j})\cdot j\omega\alpha}).
    \end{align*}
    Using $|k_1+\cdots+k_n|_s\le|k_1|_s+\cdots|k_n|_s$ and $\|k\|_1\le |k|_s$, where $k_1,\cdots,k_n,k\in\mathbb Z_*^{\mathbb N}$, $k=(k^i)_{i\in\mathbb N}$, $\|k\|_1:=\sum_i |k^i|$, we obtain 
    \begin{align*}
        \|\Psi[\hat h]\|_{\rho}& = \sum_{u\in\mathbb Z_*^{\mathbb N}}e^{\rho|u|_s}|\widehat{\Psi[\hat h]}_u|\\
        &\le \sum_{u\in\mathbb Z_*^{\mathbb N}}e^{\rho|u|_s}\sum_{n_0,\cdots,n_L\in\mathbb N}(\prod_{j=0}^L\frac1{n_j!})\\&\quad\sum_{\sum_{t=0}^L(k_t+k_{t,1}+\cdots+k_{t,n_t})=u}|\widehat H(k_0,\cdots,K_L)|(\prod_{j=0}^{L}\|k_j\|_1^{n_j}|\hat h_{k_{j,1}}|\cdots|\hat h_{k_{j,n_j}}|)\\
        &\le\sum_{u\in\mathbb Z_*^{\mathbb N}}\sum_{n_0,\cdots,n_L\in\mathbb N}\sum_{\sum_{t=0}^L(k_t+k_{t,1}+\cdots+k_{t,n_t})=u}e^{\rho\sum_{t=0}^L(|k_t|_s+|k_{t,1}|_s+\cdots+|k_{t,n_t}|_s)}|\widehat H(k_0,\cdots,k_L)|\\
        &\quad\quad\cdot(\prod_{j=0}^L\frac{|k_j|_s^{n_j}}{n_j!}|\hat h_{k_{j,1}}|\cdots|\hat h_{k_{j,n_j}}|)\\
        &\le\sum_{k_0,\cdots,k_L\in\mathbb Z_*^{\mathbb N}}e^{\sum_{t=0}^L\rho|k_t|_s}|\widehat H(k_0,\cdots,k_L)|\sum_{n_0,\cdots,n_L\in\mathbb N}\\
        &\quad\quad(\prod_{j=0}^L\sum_{k_{j,1},\cdots,k_{j,n_j}\in\mathbb Z_*^{\mathbb N}}e^{\rho|k_{j,1}|_s+\cdots+\rho|k_{j,n_j}|_s}\frac{|k_j|_s^{n_j}}{n_j!}|\hat h_{k_{j,1}}|\cdots|\hat h_{k_{j,n_j}}|)\\
        &\le\sum_{k_0,\cdots,k_L\in\mathbb Z_*^{\mathbb N}}e^{\sum_{t=0}^L\rho|k_t|_s}|\widehat H(k_0,\cdots,k_L)|\sum_{n_0,\cdots,n_L\in\mathbb N}\prod_{j=0}^L\frac{|k_j|_s^{n_j}}{n_j!}\|\hat h\|_{\rho}^{n_j}\\
        &\le\sum_{k_0,\cdots,k_L\in\mathbb Z_*^{\mathbb N}}e^{\sum_{t=0}^L\rho|k_t|_s}|\widehat H(k_0,\cdots,k_L)|e^{\sum_{j=0}^L|k_j|_s\|\hat h\|_{\rho}}\\
        &=\|\widehat H\|_{L,\rho+\|\hat h\|_\rho}.
    \end{align*}

\end{proof}

\section{Short range Models with almost-periodic potential}\label{S4}
\subsection{Short range models}
Short range models are known as Frenkel-Kontorova models. In these models, we consider the particle only interacts with the nearest particles. The configuration of the system is described by $u=\{u_n\}_{n\in\mathbb Z}$ with $u_n\in\mathbb R$. The physical meaning of $u_n$ is the position of the particles. And we assumed the system is located in the almost-periodic potential field about position. Then, we can get the formal energy of the system (\ref{short}).

To obtain the critical points of the formal energy, we take the formal derivative of $\frac{\partial}{\partial u_n}\mathscr S(\{u_n\}_{n\in\mathbb Z})=0$, that is
\begin{equation}\label{short-equi}
u_{n+1}+u_{n-1}-2u_n+V'(u_n)=0
\end{equation}
In this way, our goal is to solve the equilibrium equation(\ref{short-equi}).

\subsection{Hull function and external forces}
Let the rotation number $\omega\in\mathbb R$, and consider the solution of the form $$u_n=h(n\omega),\ \forall n\in\mathbb Z.$$
$$h(\theta)=\theta+\tilde h(\theta),$$where $\theta:=n\omega$ and $\tilde h$ is an almost-periodic function with the basic frequency $\alpha\in[0,1]^{\mathbb N}$. The function $h$ is often referred as `hull' function of the configuration in solid state physics.

We can write $\tilde h$ in the form of Fourier series,$$\tilde h(\theta)=\sum_{k\in\mathbb Z_*^\mathbb{N}}\hat h_ke^{ik\cdot\theta\alpha}.$$Denote $\sigma:=\theta\alpha$, then $$\tilde h(\theta)=\sum_{k\in\mathbb Z_*^\mathbb{N}}\hat h_ke^{ik\cdot\sigma}=:\hat h(\sigma),$$where $\hat h:\mathbb T^{\mathbb N}\rightarrow\mathbb R$. Let $AP(\alpha)$ be the set of functions of the form of $\tilde h$. 

Denote $\partial_\alpha:=\alpha\cdot\nabla$, $\widehat V(\theta\alpha)=V(\theta)$, then equation (\ref{short-equi}) is equivalent to
\begin{equation}\label{2}
h(\theta+\omega)+h(\theta-\omega)-2h(\theta)+\partial_\alpha\widehat V(\alpha h(\theta))=0,
\end{equation}
or equivalently
\begin{equation}\label{3}
\hat h(\sigma+\omega\alpha)+\hat h(\sigma-\omega\alpha)-2\hat h(\sigma)+\partial_\alpha\widehat V(\sigma+\alpha\hat h(\sigma))=0.
\end{equation}

To solve equation (\ref{3}), we add an external parameter and define the below operator,
$$\mathcal E[\hat h,\lambda](\theta):=\hat h(\sigma+\omega\alpha)+\hat h(\sigma-\omega\alpha)-2\hat h(\sigma)+\widehat U(\sigma+\alpha\hat h(\sigma))+\lambda,$$
where $\widehat U:\mathbb T^{\mathbb N}\rightarrow \mathbb R,\ \lambda\in\mathbb R$.

Obviously, if we can solve the general case
\begin{equation}\label{4}
\mathcal E[\hat h,\lambda]=\hat h(\sigma+\omega\alpha)+\hat h(\sigma-\omega\alpha)-2\hat h(\sigma)+\widehat U(\sigma+\alpha\hat h(\sigma))+\lambda=0,
\end{equation}
then we can easily solve the particular case equation (\ref{3}), just let $\widehat U=\partial_\alpha \widehat V$, $\lambda=0$ and we can get it. Later, we will prove the Vanishing lemma which indicates that if $\widehat U=\partial_\alpha \widehat V$, then $\lambda=0$.

If $[\hat h(\sigma),\lambda]$ is a solution of equation (\ref{4}), it's easy to check that $[\hat h(\sigma+\beta\alpha)+\beta,\lambda]$ is also a solution for any $\beta\in\mathbb R$. Therefore, by choosing $\beta$ we can always choose our solution normalized in such a way that $$\mathop{\rm lim}\limits_{T\rightarrow\infty}\frac1{2T}\int_{-T}^T\tilde h(\theta)d\theta\equiv\int_{\mathbb T^{\mathbb N}}\hat h(\sigma)d\sigma=0.$$ Note that the choice of $\beta$ is unique. Therefore, we always assume $\hat h$ is normalized.

\subsection{The main theorem of short range models}
\begin{thm}[Short range KAM theorem]\label{sr}
Let $h(\theta)=\theta+\tilde h(\theta)$, $\tilde h(\theta)=\hat h(\alpha\theta)=\sum_{k\in\mathbb Z^\mathbb{N}_*}\hat h_ke^{ik\cdot\alpha\theta}$, $\hat h_0=0$, $\hat h\in\mathscr A_{\rho_0}^1$. $\alpha\in[0,1]^{\mathbb N}$ is rationally independent. $\widehat U\in\mathscr A_{\rho_0+\|\hat h\|_{\rho_0}+\iota}^2$, $\iota>0$ Denote $\hat l=1+\partial_\alpha\hat h$, $T_{x}(\sigma)=\sigma+x$, We assume the following: 
\begin{itemize}
\item[{\rm (H1)}]
Diophantine condition: $|\omega\alpha\cdot k-2n\pi|\ge\frac\nu{\prod_{j\in\mathbb N}(1+\langle\langle j\rangle\rangle^{1+\tau}|k_j|^{1+\tau})}$ holds for $\forall k\in\mathbb Z^\mathbb{N}_*\setminus\{0\}$, $\forall n\in\mathbb Z$, where $\tau,\nu>0$.
\item[{\rm (H2)}]
Non-degeneracy condition: $\|\hat l(\sigma)\|_{\rho_0}\le N^+$, $\|(\hat l(\sigma))^{-1}\|_{\rho_0}\le N^-$, $\big|\big<\frac1{\hat l\cdot\hat l\circ T_{-\omega\alpha}}\big>\big|\ge c>0$, where $N^+,N^-,c$ are called the condition numbers. 
\end{itemize}
If $\|\mathcal E[\hat h,\lambda]\|_{\rho_0}$ is small enough $(\|\mathcal E[\hat h,\lambda]\|_{\rho_0}\le\epsilon\le\epsilon^*(N^+,N^-,c,\tau,\nu,\iota,\rho_0,\|\widehat U\|_{\mathscr A_{\rho_0+\|\hat h\|_{\rho_0}+\iota}^2})$, then there exists an analytic function $\hat h^*\in\mathscr A_{\frac{\rho_0}{2}}$ and $\lambda^*\in\mathbb R$, such that 
$$\mathcal E[\hat h^*,\lambda^*]=0.$$
Moreover, $$\|\hat h-\hat h^*\|_\frac{\rho_0}{2}\le C_1(N^+,N^-,c,\tau,\nu,\iota,\rho_0,\|\widehat U\|_{\mathscr A_{\rho_0+\|\hat h\|_{\rho_0}+\iota}^2})\epsilon_0,$$
$$|\lambda-\lambda^*|\le C_2(N^+,N^-,c,\tau,\nu,\iota,\rho_0,\|\widehat U\|_{\mathscr A_{\rho_0+\|\hat h\|_{\rho_0}+\iota}^2})\epsilon_0.$$
The solution $[\hat h^*,\lambda^*]$ is the only solution of $\mathcal{E}[\hat h^*,\lambda^*]=0$ with zero average for $\hat h^*$ in a ball centered at $\hat h$ in $\mathscr A_{\frac{3\rho_0}{8}}$, i.e. $[\hat h^*,\lambda^*]$ is the unique solution in the set
$$\{ \hat g\in\mathscr A_{\frac{3}{8}\rho_0}: \langle \hat g\rangle=0, \|\hat g-\hat h\|_{\frac{3}{8}\rho_0}\le C_3(N^+,N^-,c,\tau,\nu,\iota,\rho_0,\|\widehat U\|_{\mathscr A_{\rho_0+\|\hat h\|_{\rho_0}+\iota}^2})  \},$$
where $C_1$, $C_2$, $C_3$ will be made explicit along the proof.
\end{thm}

\subsection{Proof of  Theorem~\ref{sr}}
\subsubsection{Quasi-Newton iteration}
Consider the equation
\begin{equation}\label{5}
\mathcal E[\hat h,\lambda]=e,
\end{equation}
where $e$ is a ``small" function of $\theta$.

By the classical Newton iterative method, we consider the following equation:
\begin{equation}\label{6}
\widehat\Delta(\sigma+\omega\alpha)+\widehat\Delta(\sigma-\omega\alpha)-2\widehat\Delta(\sigma)+\partial_\alpha\widehat U(\sigma+\alpha\hat h(\sigma))\widehat\Delta(\sigma)+\delta=-e.
\end{equation}
If we can solve equation (\ref{6}), then $[\hat h+\widehat\Delta,\lambda+\delta]$ will be a better approximate solution of equation (\ref{4}). However, equation (\ref{6}) is hard to solve due to the term $\partial_\alpha\widehat U(\sigma+\alpha\hat h(\sigma))$ which doesn't have constant coefficients. Therefore we consider the following equation which is obtained by taking derivative of equation (\ref{5}), hoping to eliminate the term $\partial_\alpha\widehat U(\sigma+\alpha\hat h(\sigma))$:
\begin{equation}\label{7}
\partial_\alpha\hat h(\sigma+\omega\alpha)+\partial_\alpha\hat h(\sigma-\omega\alpha)-2\partial_\alpha\hat h(\sigma)+\partial_\alpha\widehat U(\sigma+\alpha\hat h(\sigma))(1+\partial_\alpha\hat h(\sigma))=e'(\theta).
\end{equation}

Denote $\hat l(\sigma)=1+\partial_\alpha\hat h(\sigma)$, then equation (\ref{7}) is equivalent to 
\begin{equation}\label{8}
\hat l(\sigma+\omega\alpha)+\hat l(\sigma-\omega\alpha)-2\hat l(\sigma)+\partial_\alpha\widehat U(\sigma+\alpha\hat h(\sigma))\hat l(\sigma)=e'(\theta).
\end{equation}
Substituting equation (\ref{8}) into equation (\ref{6}) and eliminating the term $\partial_\alpha\widehat U(\sigma+\alpha\hat h(\sigma))$, we obtain that
\begin{equation*}
\widehat\Delta(\sigma+\omega\alpha)+\widehat\Delta(\sigma-\omega\alpha)+\frac{e'(\theta)-\hat l(\sigma+\omega\alpha)-\hat l(\sigma-\omega\alpha)}{\hat l(\sigma)}\widehat\Delta(\sigma)=-e-\delta.
\end{equation*}

However, this equation is still hard to solve. As we can see that the term $e'(\theta)\widehat\Delta(\sigma)$ is small, we omit this term and obtain that
\begin{equation}\label{9}
\widehat\Delta(\sigma+\omega\alpha)+\widehat\Delta(\sigma-\omega\alpha)-\frac{\hat l(\sigma+\omega\alpha)+\hat l(\sigma-\omega\alpha)}{\hat l(\sigma)}\widehat\Delta(\sigma)=-e-\delta,
\end{equation}
which is called quasi-Newton equation.

We observe that equation (\ref{9}) is equivalent to the following system:
\begin{align}
\left(\frac{\widehat\Delta}{\hat l}\right)\circ T_{-\omega\alpha}-\left(\frac{\widehat\Delta}{\hat l}\right) &=\frac{\widehat W}{\hat l\cdot\hat l\circ T_{-\omega\alpha}},\label{10}\\
\widehat W\circ T_{\omega\alpha}-\widehat W &=\hat l\cdot(e+\delta).\label{11}
\end{align}

We write $\widehat W=\widehat W^0+\overline{\widehat W}$, where $\widehat W^0$ is a function with zero average and $\overline{\widehat W}$ is a constant, the average of $\widehat W$. Both equation (\ref{10}) and equation (\ref{11}) are homological equations. To use Lemma \ref{l1}, we should make sure the right hand side of equation (\ref{10}) and (\ref{11}) have zero average. Therefore, by equation (\ref{10}) we obtain
$$\overline{\widehat W}=-\frac{\big<\frac{\widehat W^0}{\hat l\cdot\hat l\circ T_{-\omega\alpha}}\big>}{\big<\frac1{\hat l\cdot\hat l\circ T_{-\omega\alpha}}\big>},$$
and by equation (\ref{11}) we obtain
$$\delta=-\frac{\big<\hat l\cdot e\big>}{\big<\hat l\big>}.$$

By the Fourier expansion of $\hat h$ and $\hat l=1+\partial_\alpha\hat h$, we obtain that $\hat l_k=\delta_{k,0}+ik\cdot\alpha\hat h_k$, where $\delta_{k,0}$ is the Kronecker sign. Therefore $\big<\hat l\big>=1$, $\delta=-\big\langle\hat l\cdot e\big\rangle.$

Then, we can apply Lemma (\ref{l1}) to solve equation (\ref{10}) and (\ref{11}), and estimate the solutions. We assume that $\tilde\beta$ is a solution of equation (\ref{10}), then $\tilde\beta+C$ is also a solution of equation (\ref{10}), where $C$ is an arbitrary constant. We choose $\bar\beta:=-\big<\tilde\beta\cdot\hat l\big>$, then $\tilde\beta+\bar\beta$ is also a solution of equation (\ref{10}). Therefore, $\widehat\Delta=(\tilde\beta+\bar\beta)\cdot\hat l$ and $\big<\widehat\Delta\big>=0$.

We hope that $[\hat h+\widehat\Delta,\lambda+\delta]$ is a better approximate solution of equation (\ref{4}), so we compute $\mathcal E[\hat h+\widehat\Delta,\lambda+\delta]$:
\begin{equation*}
\begin{split}
\mathcal E[\hat h+\widehat\Delta,&\lambda+\delta]\\
&=\mathcal E[\hat h,\lambda]+\widehat\Delta(\sigma+\omega\alpha)+\widehat\Delta(\sigma-\omega\alpha)-2\widehat\Delta(\sigma)+\delta\\
&\quad+\widehat U(\sigma+\alpha(\hat h+\widehat\Delta)(\sigma))-\widehat U(\sigma+\alpha\hat h(\sigma))\\
&=e+(-e)+\frac{\hat l(\sigma+\omega\alpha)+\hat l(\sigma-\omega\alpha)-2\hat l(\sigma)}{\hat l(\sigma)}\widehat\Delta(\sigma)\\
&\quad+\widehat U(\sigma+\alpha(\hat h+\widehat\Delta)(\sigma))-
\widehat U(\sigma+\alpha\hat h(\sigma))\\
&=\frac{e'(\theta)-\partial_\alpha\widehat U(\sigma+\alpha\hat h(\sigma))\hat l(\sigma)}{\hat l(\sigma)}\widehat\Delta(\sigma)+\widehat U(\sigma+\alpha(\hat h+\widehat\Delta)(\sigma))-
\widehat U(\sigma+\alpha\hat h(\sigma))\\
&=e'\frac{\widehat\Delta(\sigma)}{\hat l(\sigma)}+R,
\end{split}
\end{equation*}
where $R=\widehat U(\sigma+\alpha(\hat h+\widehat\Delta)(\sigma))-
\widehat U(\sigma+\alpha\hat h(\sigma))-\partial_\alpha\widehat U(\sigma+\alpha\hat h(\sigma))\widehat\Delta(\sigma)$.

\subsubsection{Estimates for one iterative step}
\paragraph{\textbf{Estimates for solutions}}
For equation (\ref{11}), using Lemma \ref{l1}, for $\forall 0<\rho'<\rho$, we have
\begin{equation*}
\begin{split}
\|\widehat W^0\|_{\rho'}&\le C\nu^{-1}{\rm exp}\left(\frac{\tilde\tau}{(\rho-\rho')^\frac1s}{\rm ln}\frac{\tilde\tau}{\rho-\rho'}\right)N^+\|e+\delta\|_\rho\\
&\le C\nu^{-1}{\rm exp}\left(\frac{\tilde\tau}{(\rho-\rho')^\frac1s}{\rm ln}\frac{\tilde\tau}{\rho-\rho'}\right)N^+(\|e\|_\rho+|\big<\hat l\cdot e\big>|)\\
&\le C\nu^{-1}{\rm exp}\left(\frac{\tilde\tau}{(\rho-\rho')^\frac1s}{\rm ln}\frac{\tilde\tau}{\rho-\rho'}\right)N^+(\|e\|_\rho+N^+\|e\|_\rho)\\
&=C\nu^{-1}{\rm exp}\left(\frac{\tilde\tau}{(\rho-\rho')^\frac1s}{\rm ln}\frac{\tilde\tau}{\rho-\rho'}\right)N^+(1+N^+)\|e\|_\rho.
\end{split}
\end{equation*}

We have the estimate for $\overline{\widehat W}$:
\begin{equation*}
\begin{split}
|\overline{\widehat W}|&\le\frac1c\|\widehat W^0\|_{\rho'}(N^-)^2\\
&\le \frac1c(N^-)^2 C\nu^{-1}{\rm exp}\left(\frac{\tilde\tau}{(\rho-\rho')^\frac1s}{\rm ln}\frac{\tilde\tau}{\rho-\rho'}\right)N^+(1+N^+)\|e\|_\rho.
\end{split}
\end{equation*}

Therefore, we obtain the estimate for $\widehat W$:
\begin{equation*}
\begin{split}
\|\widehat W\|_{\rho'}&\le\|\widehat W^0\|_{\rho'}+|\overline{\widehat W}|\\
&\le\left(1+\frac1c(N^-)^2\right)C\nu^{-1}{\rm exp}\left(\frac{\tilde\tau}{(\rho-\rho')^\frac1s}{\rm ln}\frac{\tilde\tau}{\rho-\rho'}\right)N^+(1+N^+)\|e\|_\rho.
\end{split}
\end{equation*}

Similarly, for equation (\ref{10}), using Lemma \ref{l1}, for $\forall 0<\rho''<\rho'$, we have
\begin{equation*}
\begin{split}
\|\tilde\beta\|_{\rho''}&\le C\nu^{-1}{\rm exp}\left(\frac{\tilde\tau}{(\rho'-\rho'')^\frac1s}{\rm ln}\frac{\tilde\tau}{\rho'-\rho''}\right)(N^-)^2\|\widehat W\|_{\rho'}\\
&\le\left(1+\frac1c(N^-)^2\right)C^2\nu^{-2}N^+(1+N^+)(N^-)^2\\
&\quad\cdot{\rm exp}\left(\frac{\tilde\tau}{(\rho'-\rho'')^\frac1s}{\rm ln}\frac{\tilde\tau}{\rho'-\rho''}+\frac{\tilde\tau}{(\rho-\rho')^\frac1s}{\rm ln}\frac{\tilde\tau}{\rho-\rho'}\right)\|e\|_\rho.
\end{split}
\end{equation*}
Particularly, if we take $\rho-\rho'=\rho'-\rho''$, we have
\begin{equation*}
\begin{split}
\|\tilde\beta\|_{\rho''}&\le\left(1+\frac1c(N^-)^2\right)C^2\nu^{-2}N^+(1+N^+)(N^-)^2\\
&\quad\cdot{\rm exp}\left(2\frac{\tilde\tau}{(\frac{\rho-\rho''}2)^\frac1s}{\rm ln}\frac{2\tilde\tau}{\rho-\rho''}\right)\|e\|_\rho\\
&=\left(1+\frac1c(N^-)^2\right)C^2\nu^{-2}N^+(1+N^+)(N^-)^2\\
&\quad\cdot{\rm exp}\left(2^{1+\frac1s}\frac{\tilde\tau}{(\rho-\rho'')^\frac1s}{\rm ln}\frac{2\tilde\tau}{\rho-\rho''}\right)\|e\|_\rho,
\end{split}
\end{equation*}
and
\begin{equation}\label{12}
\begin{split}
\|\widehat\Delta\|_{\rho''}&=\|(\tilde\beta+\bar\beta)\cdot\hat l\|_{\rho''}\le N^+(\|\tilde\beta\|_{\rho''}+|\bar\beta|)\\
&=N^+\left(\|\tilde\beta\|_{\rho''}+\big|\big<\tilde\beta\cdot\hat l\big>\big|\right)\le N^+(1+N^+)\|\tilde\beta\|_{\rho''}\\
&\le\left(1+\frac1c(N^-)^2\right)C^2\nu^{-2}(N^+)^2(1+N^+)^2(N^-)^2\\
&\quad\cdot{\rm exp}\left(2^{1+\frac1s}\frac{\tilde\tau}{(\rho-\rho'')^\frac1s}{\rm ln}\frac{2\tilde\tau}{\rho-\rho''}\right)\|e\|_\rho.
\end{split}
\end{equation}

Using Lemma \ref{l4}, for $\forall 0<\rho'''<\rho''$, we have
\begin{equation*}
\begin{split}
\|\widehat\Delta'\|_{\rho'''}&\le(\rho''-\rho''')^{-1}\|\widehat\Delta\|_{\rho''}\\
&\le\left(1+\frac1c(N^-)^2\right)C^2\nu^{-2}(N^+)^2(1+N^+)^2(N^-)^2\\
&\quad\cdot(\rho''-\rho''')^{-1}{\rm exp}\left(2^{1+\frac1s}\frac{\tilde\tau}{(\rho-\rho'')^\frac1s}{\rm ln}\frac{2\tilde\tau}{\rho-\rho''}\right)\|e\|_\rho.
\end{split}
\end{equation*}

In order to make sure $\sigma+\alpha(\hat h(\sigma)+\widehat\Delta(\sigma))$ is still in the domain of $\widehat U$, from the definition of $\iota$ we can see that it is sufficient to let
$\|\widehat\Delta\|_{\rho''}<\iota$. 
Later we will show that this can hold for sufficiently small $\|e\|_\rho$, using the estimate (\ref{12}). According to Lemma \ref{l5}, we have
$$\|R\|_{\rho''}\le\|\widehat U\|_{\mathscr A_{\rho+\|\hat h\|_{\rho}+\iota}^2}\|\widehat\Delta\|_{\rho''}^2.$$

By now, we can estimate $\mathcal E[\hat h+\widehat\Delta,\lambda+\delta]$:
\begin{equation*}
\begin{split}
\|\mathcal E[\hat h&+\widehat\Delta,\lambda+\delta]\|_{\rho''}\le\|e'\frac{\widehat\Delta}{\hat l}\|_{\rho''}+\|R\|_{\rho''}\\
&\le (\rho-\rho'')^{-1}\|e\|_\rho N^-\|\widehat\Delta\|_{\rho''}+\|\widehat U\|_{\mathscr A_{\rho+\|\hat h\|_{\rho}+\iota}^2}\|\widehat\Delta\|_{\rho''}^2\\
&\le (\rho-\rho'')^{-1}\left(1+\frac1c(N^-)^2\right)C^2\nu^{-2}(N^+)^2(1+N^+)^2(N^-)^3{\rm exp}\left(2^{1+\frac1s}\frac{\tilde\tau}{(\rho-\rho'')^\frac1s}{\rm ln}\frac{2\tilde\tau}{\rho-\rho''}\right)\|e\|_\rho^2\\
&\quad+\|\widehat U\|_{\mathscr A_{\rho+\|\hat h\|_{\rho}+\iota}^2}\left(1+\frac1c(N^-)^2\right)^2C^4\nu^{-4}(N^+)^4(1+N^+)^4(N^-)^4{\rm exp}\left(2^{2+\frac1s}\frac{\tilde\tau}{(\rho-\rho'')^\frac1s}{\rm ln}\frac{2\tilde\tau}{\rho-\rho''}\right)\|e\|_\rho^2\\
&\le M'\left(1+(\rho-\rho'')^{-1}\right){\rm exp}\left(2^{2+\frac1s}\frac{\tilde\tau}{(\rho-\rho'')^\frac1s}{\rm ln}\frac{2\tilde\tau}{\rho-\rho''}\right)\|e\|_\rho^2,
\end{split}
\end{equation*}
where $M'=\left(1+\frac1c(N^-)^2\right)C^2\nu^{-2}(N^+)^2(1+N^+)^2(N^-)^3+\|\widehat U\|_{\mathscr A_{\rho+\|\hat h\|_{\rho}+\iota}^2}\left(1+\frac1c(N^-)^2\right)^2C^4\nu^{-4}(N^+)^4(1+N^+)^4(N^-)^4$. 

Therefore, when
$\|\widehat\Delta\|_{\rho''}<\iota$ holds 
, we have the estimate for one iterative step:
\begin{equation}\label{13}
\|\mathcal E[\hat h+\widehat\Delta,\lambda+\delta]\|_{\rho''}\le M'\left(1+(\rho-\rho'')^{-1}\right){\rm exp}\left(2^{2+\frac1s}\frac{\tilde\tau}{(\rho-\rho'')^\frac1s}{\rm ln}\frac{2\tilde\tau}{\rho-\rho''}\right)\|\mathcal E[\hat h,\lambda]\|_{\rho}^2.
\end{equation}

\paragraph{\textbf{Estimates for the condition numbers}}
Let $N^+(\hat h,\rho_0):=\|1+\partial_\alpha\hat h\|_{\rho_0},\ N^-(\hat h,\rho_0):=\|(1+\partial_\alpha\hat h)^{-1}\|_{\rho_0},\ c(\hat h):=\left|\big<\frac1{\hat l\cdot\hat l\circ T_{-\omega\alpha}}\big>\right|$, then
\begin{equation*}
\begin{split}
N^+(\hat h,\rho_0)&=\|1+\partial_\alpha\hat h\|_{\rho_0}\\
&=\|1+\partial_\alpha\hat h_0+\partial_\alpha(\hat h-\hat h_0)\|_{\rho_0}\\
&\le\|1+\partial_\alpha\hat h_0\|_{\rho_0}+\|\partial_\alpha(\hat h-\hat h_0)\|_{\rho_0}\\
&=N^+(\hat h_0,\rho_0)+\|\partial_\alpha(\hat h-\hat h_0)\|_{\rho_0},\\
N^-(\hat h,\rho_0)&=\|(1+\partial_\alpha\hat h)^{-1}\|_{\rho_0}\\
&=\left\|\frac{1+\partial_\alpha\hat h_0}{(1+\partial_\alpha\hat h)(1+\partial_\alpha\hat h_0)}\right\|_{\rho_0}\\
&=\left\|\frac{1+\partial_\alpha\hat h+\partial_\alpha(\hat h_0-\hat h)}{(1+\partial_\alpha\hat h)(1+\partial_\alpha\hat h_0)}\right\|_{\rho_0}\\
&\le N^-(\hat h_0,\rho_0)+\|\partial_\alpha(\hat h-\hat h_0)\|_{\rho_0}N^-(\hat h_0,\rho_0)N^-(\hat h,\rho_0),
\end{split}
\end{equation*}

\begin{equation*}
\begin{split}
|c(\hat h)-c(\hat h_0)|&\le\left|\big<\frac1{\hat l\cdot\hat l\circ T_{-\omega\alpha}}\big>-\big<\frac1{\hat l_0\cdot\hat l_0\circ T_{-\omega\alpha}}\big>\right|\\
&=\left|\big<\frac{\hat l_0\cdot(\hat l_0-\hat l)\circ T_{-\omega\alpha}+(\hat l_0-\hat l)\cdot\hat l\circ T_{-\omega\alpha}}{\hat l_0\cdot\hat l_0\circ T_{-\omega\alpha}\cdot\hat l\cdot\hat l\circ T_{-\omega\alpha}}\big>\right|\\
&\le\left(N^-(\hat h,\rho_0)N^-(\hat h_0,\rho_0)\right)^2(\|\hat l\|_{\rho_0}+\|\hat l_0\|_{\rho_0})\|\hat l-\hat l_0\|_{\rho_0}.
\end{split}
\end{equation*}

Therefore, there exists a constant $\gamma>0$, such that when $\|\hat h-\hat h_0\|_{\mathscr A_\rho^1}<\gamma$, for $\forall ~0<\rho<\rho_0$, there's always $N^{\pm}(\hat h,\rho)\le N^{\pm}(\hat h,\rho_0)\le2N^{\pm}(\hat h_0,\rho_0)$, $c(\hat h)\ge\frac{1}{2}c(\hat h_0)$. 

\subsubsection{Estimates for the iterative steps}
Choose $\rho_0\le8$. Let
\begin{equation*}
\rho_n=\rho_{n-1}-\frac{\rho_0}42^{-n}=\rho_0\left(1-\frac14\sum_{j=1}^n2^{-j}\right),
\end{equation*}
then $\rho_n\rightarrow\frac34\rho_0\ (n\rightarrow\infty)$. 

We denote $[\hat h_0,\lambda_0]:=[\hat h,\lambda]$. In the first iterative step, we obtain $[\widehat\Delta_0,\delta_0]$. We denote $[\hat h_1,\lambda_1]:=[\hat h_0+\widehat\Delta_0,\lambda_0+\delta_0]$. Similarly, we obtain a sequence of approximate solutions $\{[\hat h_n,\lambda_n]\}_{n\in\mathbb N}$, and we denote $\epsilon_n=\|\mathcal E[\hat h_n,\lambda_n]\|_{\rho_n}$. 

We notice that, in the first iterative step, we can choose a sufficiently small $\bar\delta_0>0$, such that when $\epsilon_0\in(0,\bar\delta_0)$, $\widehat U$ is well-defined and $N^{\pm}(\hat h_1,\rho_1)\le N^{\pm}(\hat h_1,\rho_0)\le2N^{\pm}(\hat h_0,\rho_0)$, $c(\hat h_1)\le2c(\hat h_0)$. Since $M'$ is a function of condition numbers, its value changes during the iteration. Provided that there exists a constant $M$ such that $M'<M$ holds uniformly, we can prove the convergence of the iteration inductively. 

Following from the notations we give earlier, we have the following estimates:
\begin{equation*}
\begin{split}
\epsilon_n&\le M\left(1+\frac{2^{n+2}}{\rho_0}\right){\rm exp}\left(2^{2+\frac1s}\frac{2^{\frac{n+2}{s}}\tilde\tau}{{\rho_0}^\frac1s}{\rm ln}\frac{2^{n+3}\tilde\tau}{\rho_0}\right)\epsilon^2_{n-1}\le\frac M{\rho_0}2^{n+3}{\rm exp}\left(2^{2+\frac1s}\frac{2^{\frac{n+2}{s}}\tilde\tau}{{\rho_0}^\frac1s}{\rm ln}\frac{2^{n+3}\tilde\tau}{\rho_0}\right)\epsilon^2_{n-1}\\
&\le\left(\frac{M}{\rho_0}\right)^{1+2+\cdots+2^{n-1}}2^{(n+3)+2(n+2)+\cdots+2^{n-1}\cdot4}\\
&\quad\cdot{\rm exp}\left(\frac{2^{2+\frac1s}\tilde\tau}{{\rho_0}^\frac1s}\left(2^{\frac{n+2}{s}}{\rm ln}\frac{2^{n+3}\tilde\tau}{\rho_0}+2^{1+\frac{n+1}{s}}{\rm ln}\frac{2^{n+2}\tilde\tau}{\rho_0}+\cdots+2^{n-1+\frac3s}{\rm ln}\frac{2^4\tilde\tau}{\rho_0}\right)\right)\epsilon^{2^n}_{0}\\
&\le\left(\frac{M}{\rho_0}\right)^{2^n}2^{2^{n+3}}{\rm exp}\left(\frac{2^{2+\frac1s}\tilde\tau}{{\rho_0}^\frac1s}\left((n+3)2^{\frac{n+2}{s}}+(n+2)2^{1+\frac{n+1}{s}}+\cdots+4\cdot2^{n-1+\frac3s}\right){\rm ln}2\right)\\
&\quad\cdot{\rm exp}\left(\frac{2^{2+\frac1s}\tilde\tau}{{\rho_0}^\frac1s}\left(2^{\frac{n+2}{s}}+2^{1+\frac{n+1}{s}}+\cdots+2^{n-1+\frac3s}\right){\rm ln}\frac{\tilde\tau}{\rho_0}\right)\epsilon^{2^n}_{0}\\
&\le\left(\frac{M}{\rho_0}\right)^{2^{n}}2^{2^{n+3}}2^{2^{n+6}\frac{2^{2+\frac1s}\tilde\tau}{{\rho_0}^\frac1s}}\left(\frac{\tilde\tau}{\rho_0}\right)^{\frac{2^{2+\frac1s}\tilde\tau}{\rho_0^\frac1s}2^{n+3}}\epsilon_0^{2^n}\le\left(\frac{M}{\rho_0}2^82^\frac{2^{8+\frac1s}\tilde\tau}{{\rho_0}^\frac1s}\left(\frac{\tilde\tau}{\rho_0}\right)^{\frac{2^{5+\frac1s}\tilde\tau}{\rho_0^{\frac1s}}}\epsilon_0\right)^{2^n}\\
&=:(A\epsilon_0)^{2^n},
\end{split}
\end{equation*}
where $A=\frac{M}{\rho_0}2^82^\frac{2^{8+\frac1s}\tilde\tau}{{\rho_0}^\frac1s}\left(\frac{\tilde\tau}{\rho_0}\right)^{\frac{2^{5+\frac1s}\tilde\tau}{\rho_0^{\frac1s}}}$, $n\in\mathbb N$. Hence, if $A\epsilon_0<1$, $\epsilon_n$ decreases faster than any exponential.

Now we show that $\widehat U$ is well-defined during the iteration. Let 
$\left(1+\frac1c(N^-)^2\right)C^2\nu^{-2}(N^+)^2(1+N^+)^2(N^-)^2$
 be uniformly bounded by $L$ and $s>2$, then
\begin{equation*}
\begin{split}
\sum_{j=0}^n\|\widehat\Delta_j\|_{\rho_n}&\le\sum_{j=0}^n\|\widehat\Delta_j\|_{\rho_j+1}\le\sum_{j=0}^nL{\rm exp}\left(2^{1+\frac1s}\frac{2^{\frac{j+3}{s}}\tilde\tau}{\rho_0^\frac1s}{\rm ln}\frac{2^{j+4}\tilde\tau}{\rho_0}\right)\epsilon_j\\
&\le\sum_{j=0}^nL{\rm exp}\left(2^{1+\frac1s}\frac{2^{\frac{j+3}{s}}\tilde\tau}{\rho_0^\frac1s}\left(\frac{2^{j+4}\tilde\tau}{\rho_0}\right)^{\frac12}\right)(A\epsilon_0)^{2^j}\\
&\le\sum_{j=0}^nL{\rm exp}\left(2\tilde\tau^{\frac32}\left(\frac{2^{j+4}}{\rho_0}\right)^{\frac1s+\frac12}\right)(A\epsilon_0)^{2^j}\\
&\le\sum_{j=0}^nL{\rm exp}\left(2\tilde\tau^\frac32\frac{16}{\rho_0}2^j\right)(A\epsilon_0)^{2^j}=\sum_{j=0}^nL(e^{2\tilde\tau^\frac32\frac{16}{\rho_0}}A\epsilon_0)^{2^j}\\
&=:\sum_{j=0}^nL(A'\epsilon_0)^{2^j},
\end{split}
\end{equation*}
where $A'=e^{2\tilde\tau^\frac32\frac{16}{\rho_0}}A$. Hence if $\epsilon_0<\min\{\frac{\iota}{8LA'},\frac{1}{2A'}\}$, we have $\sum_{j=0}^n \|\widehat\Delta_j\|_{\rho_n}<\frac{\iota}{4},\,\forall n\in\mathbb N$, that is, $\widehat U$ is well-defined.

Finally we show that the condition numbers are uniformly bounded to finish the inductive proof. We only take $N^+$ for an example, the cases of $N^-$ and $c$ are similar. 
\begin{equation*}
\begin{split}
N^+(\hat h_n,\frac{\rho_0}{2})&\le N^+(\hat h_{n-1},\frac{\rho_0}{2})+\|\partial_\alpha(\hat h_n-\hat h_{n-1})\|_{\rho_{n+1}}\\
&\le N^+(\hat h_{n-1},\rho_{n-1})+(\rho_{n}-\rho_{n+1})^{-1}\|\widehat\Delta_{n-1}\|_{\rho_{n}}\\
&= N^+(\hat h_{n-1},\rho_{n-1})+\frac{4}{\rho_0}2^{n+1}\|\widehat\Delta_{n-1}\|_{\rho_{n}}\\
&\le N^+(\hat h_0,\rho_0)+\frac{8}{\rho_0}\sum_{j=1}^{n}2^{j}L(A'\epsilon_0)^{2^{j-1}}\\
&\le N^+(\hat h_0,\rho_0)+\frac{8}{\rho_0}\sum_{j=1}^{n}L(2A'\epsilon_0)^{2^{j-1}}.
\end{split}
\end{equation*}
Hence if $\epsilon_0\le\min\{ \frac{\rho_0N^+(h_0,\rho_0)}{32LA'}, \frac{1}{4A'} \}$, $N^+(\hat h_n,\rho_n)\le2 N^+(\hat h_0,\rho_0)$ holds for $\forall n\in\mathbb N$. Therefore, condition numbers are uniformly bounded and we have finished the inductive proof. 

Now we estimate $[\hat h^*,\lambda^*]$. For $\hat h^*$, we have
\begin{equation*}
\begin{split}
\|\hat h_N-\hat h_0\|_{\frac{\rho_0}{2}}&\le\sum_{n=1}^N\|\hat h_n-\hat h_{n-1}\|_{\frac{\rho_0}{2}}\le\sum_{n=0}^{N-1}\|\widehat\Delta_n\|_{\rho_{n+1}}\\
&\le\sum_{n=0}^{N-1}L(A'\epsilon_0)^{2^n}\le\sum_{n=1}^\infty L(A'\epsilon_0)^n\\
&=LA'\epsilon_0\frac{1}{1-A'\epsilon_0}.
\end{split}
\end{equation*}
Hence if $\epsilon_0<\frac{1}{2A'}$, we have
\begin{equation*}
\|\hat h_N-\hat h_0\|_{\frac{\rho_0}{2}}\le2LA'\epsilon_0,
\end{equation*}
therefore $\|\hat h^*-\hat h_0\|_{\frac{\rho_0}{2}}\le C_1\epsilon_0$, where $C_1=2LA'$.

Similarly, for $\lambda^*$, we have
\begin{equation*}
\begin{split}
|\lambda_N-\lambda_0|&\le\sum_{n=1}^N|\lambda_n-\lambda_{n-1}|=\sum_{n=0}^{N-1}|\delta_n|=\sum_{n=0}^{N-1}\left|\big<e_n\cdot\hat l_n\big>\right|\\
&\le\sum_{n=0}^{N-1}2N^+\|e_n\|_{\rho_n}=\sum_{n=0}^{N-1}2N^+\epsilon_n\\
&\le 2N^+\sum_{n=0}^{N-1}(A\epsilon_0)^{2^n}\le 2N^+\sum_{n=1}^\infty(A\epsilon_0)^n\\
&=2N^+A\epsilon_0\frac{1}{1-A\epsilon_0}.
\end{split}
\end{equation*}
Hence if $\epsilon_0<\frac{1}{2A}$, we have
\begin{equation*}
|\lambda_N-\lambda_0|\le4N^+A\epsilon_0,
\end{equation*}
therefore $|\lambda^*-\lambda_0|\le C_2\epsilon_0$, where $C_2=4N^+A$.

\subsubsection{Uniqueness of the solution}
Suppose that $\|\hat h^*-\hat h\|_{\frac{3}{8}\rho_0},\,\|\hat h^{**}-\hat h\|_{\frac{3}{8}\rho_0}\le r<\frac{\iota}{4}$, $\mathcal{E}[\hat h^*,\lambda^*]=\mathcal{E}[\hat h^{**},\lambda^{**}]=0$, and $\langle\hat h^*\rangle=\langle\hat h^{**}\rangle=0$. Then we have $\|\hat h^{**}-\hat h^*\|_{\frac{\rho_0}{4}}\le\|\hat h^{**}-\hat h^*\|_{\frac{3}{8}\rho_0}\le 2r$, and
\begin{equation}\label{u1}
    \begin{aligned}
    0=\mathcal{E}[\hat h^{**},\lambda^{**}]-\mathcal{E}[\hat h^*,\lambda^*]=&D_1\mathcal{E}[\hat h^*,\lambda^*](\hat h^{**}-\hat h^*) +(\lambda^{**}-\lambda^*)+\widehat U(\sigma+\alpha\hat h^{**}(\sigma))\\
    &-\widehat U(\sigma+\alpha\hat h^{*}(\sigma))-\partial_{\alpha}\widehat U(\sigma+\alpha\hat h^{*}(\sigma))(\hat h^{**}-\hat h^*)\\
    =&D_1\mathcal{E}[\hat h^*,\lambda^*](\hat h^{**}-\hat h^*)+(\lambda^{**}-\lambda^*)+R,
\end{aligned}
\end{equation}
where $R=\widehat U(\sigma+\alpha\hat h^{**}(\sigma))-\widehat U(\sigma+\alpha\hat h^{*}(\sigma))-\partial_{\alpha}\widehat U(\sigma+\alpha\hat h^{*}(\sigma))(\hat h^{**}-\hat h^*)$. By Lemma \ref{l0} and Lemma \ref{l5}, we have 
\begin{align*}
    \|R\|_{\frac{\rho_0}{4}}&\le\|\widehat U\|_{\mathscr A_{\rho_0+\|\hat h\|_{\rho_0}+\iota}^2} \| \hat h^{**}-\hat h^* \|_{\frac{\rho_0}{4}}^2\le\|\widehat U\|_{\mathscr A_{\rho_0+\|\hat h\|_{\rho_0}+\iota}^2}\| \hat h^{**}-\hat h^* \|_{\frac{\rho_0}{8}}\| \hat h^{**}-\hat h^* \|_{\frac{3}{8}\rho_0}\\
    &\le\|\widehat U\|_{\mathscr A_{\rho_0+\|\hat h\|_{\rho_0}+\iota}^2}\| \hat h^{**}-\hat h^* \|_{\frac{\rho_0}{8}}\cdot2r.
\end{align*}

Denote $\hat l^*=1+\partial_{\alpha}\hat h^*$. Since $$D_1\mathcal{E}[\hat h^*,\lambda^*]\cdot \hat l^*=\frac{d}{d\theta}\mathcal{E}[\hat h^*,\lambda^*]=0,$$
we can write the equation (\ref{u1}) as:
$$\hat l^* (D_1\mathcal{E}[\hat h^*,\lambda^*]\cdot(\hat h^{**}-\hat h^*))-(\hat h^{**}-\hat h^*)(D_1\mathcal{E}[\hat h^*,\lambda^*]\cdot\hat l^*)=-\hat l^*((\lambda^{**}-\lambda^*)+R).$$
Furthermore, we have
\begin{equation}\label{u2}
(\hat h^{**}-\hat h^*)(\sigma+\omega\alpha)+(\hat h^{**}-\hat h^*)(\sigma-\omega\alpha)-\frac{\hat l(\sigma+\omega\alpha)+\hat l(\sigma-\omega\alpha)}{\hat l(\sigma)}(\hat h^{**}-\hat h^*)(\sigma)=-(\lambda^{**}-\lambda^*)-R.
\end{equation}

Notice that the equation (\ref{u2}) and the equation (\ref{9}) share the same form. Using the uniqueness statements for the solution of equation (\ref{9}) and the estimate (\ref{12}), we conclude that
\begin{equation*}
\begin{split}
\|\hat h^{**}-\hat h^*\|_{\frac{\rho_0}{8}}&\le\left(1+\frac1c(N^-)^2\right)C^2\nu^{-2}(N^+)^2(1+N^+)^2(N^-)^2\\
&\quad\cdot{\rm exp}\left(2^{1+\frac1s}\frac{\tilde\tau}{(\frac{\rho_0}{8})^\frac1s}{\rm ln}\frac{2\tilde\tau}{\frac{\rho_0}{8}}\right)\|R\|_\frac{\rho_0}{4}\\
&\le\left(1+\frac1c(N^-)^2\right)C^2\nu^{-2}(N^+)^2(1+N^+)^2(N^-)^2\\
&\quad\cdot{\rm exp}\left(2^{1+\frac1s}\frac{\tilde\tau}{(\frac{\rho_0}{8})^\frac1s}{\rm ln}\frac{2\tilde\tau}{\frac{\rho_0}{8}}\right)\|\widehat U\|_{\mathscr A_{\rho_0+\|\hat h\|_{\rho_0}+\iota}^2}\| \hat h^{**}-\hat h^* \|_{\frac{\rho_0}{8}}\cdot 2r.
\end{split}
\end{equation*}
Therefore, when $r$ is small enough, we obtain $\hat h^{**}=\hat h^*$. This completes the proof of uniqueness of the solution in Theorem \ref{sr}.

\subsection{Vanishing lemma}
Back to the initial question, to solve equation (\ref{2}), we will prove the vanishing lemma, which shows that the addition of external force has no influence on the solving of the equation.
\begin{lem}[Vanishing lemma]
Let $[\hat h,\lambda]$ be the solution of equation (\ref{4}), $\widehat U=\partial_\alpha\widehat V$ is periodic, then $\lambda=0$.
\end{lem}
\begin{proof}
We multiply equation (\ref{4}) by $\hat l$ and compute 
$$\mathop{\rm lim}\limits_{T\rightarrow\infty}\frac{1}{2T}\int_{-T}^T\cdot d\theta$$
of all the terms.

We note that
\begin{equation*}
\lambda=-\mathop{\rm lim}\limits_{T\rightarrow\infty}\frac1{2T}\int_{-T}^T\widehat U(\alpha h(\theta))\cdot\hat l(\sigma)d\theta=-\mathop{\rm lim}\limits_{T\rightarrow\infty}\frac1{2T}(\widehat V(\alpha h(T))-\widehat V(\alpha h(-T)))=0.
\end{equation*}
In fact, we observe that 
\[
h(\theta+\omega)+h(\theta-\omega)-2h(\theta)=\tilde h(\theta+\omega)+\tilde h(\theta-\omega)-2\tilde h(\theta)\in AP(\alpha).
\]
Therefore,
\[
[h(\theta+\omega)+h(\theta-\omega)-2h(\theta)]\cdot\hat l(\sigma)\in AP(\alpha),
\]
and we have
\[
\begin{aligned}
&\lim_{T\to\infty}\frac{1}{2T}\int_{-T}^T[h(\theta+\omega)+h(\theta-\omega)-2h(\theta)]\cdot\hat l(\sigma)d\theta\\
=&\sum_{k\in\mathbb{Z}^\mathbb{N}_*\setminus\{0\}}-\hat h_k\cdot2({\rm cos}(\omega k\cdot\alpha)-1)\cdot\hat h_{-k} i(k\cdot\alpha)+\hat h_02({\rm cos}(\omega0\cdot\alpha)-1)\\
=&0,
\end{aligned}
\]
where the first equality is obtained from Fourier expansion, the second equality is obtained from the antisymmetric in $k$. Therefore, $\lambda=0$.
\end{proof}

\section{The step by step  increase of complexity method} \label{sec: finite approximation}

In this section, we present a different approach to 
the study of results of  almost periodic models. 

\subsection{Overview} 
The basic observation is that the main theorem of 
\cite{SDlL12} (reproduced here as Theorem~\ref{oldtheorem})
 is an a-posteriori theorem for  quasi-periodic 
models with a finite number of frequencies.  That is, the existence of 
an approximate solution implies the existence of a true solution. 

Given a problem involving an infinite number of frequencies,  
we can proceed inductively  considering a sequence of 
problems with more frequencies. We think of 
a model  with infinite frequencies as an infinite sum  of models,  
each one of them has a finite number of frequencies. 
We will assume that the sizes (measured in an appropriate sense) 
of the terms of the sum decreases fast enough. 

  Given an exact solution $u^N$  for
the model involving $N$ frequencies, we can consider it as 
an approximate solution for the problem with $N+1$ frequencies. 
If  the difference between the  models with
 $N$ frequencies and $N+1$ frequencies  is small (in some sense that will 
need to be made precise), then we can  use Theorem~\ref{oldtheorem} 
to construct an exact solution for the problem with $N+1$ frequencies. 
We can also control the size of the corrections. 

This is a fairly general procedure that allows to pass from finite 
dimensional a-posteriori theorems for quasi-periodic solutions to 
systems with infinitely many frequencies. The main ingredient is an
a-posteriori theorem. We do not need to go in the proof of
the theorem but rather use it as a building block.
Hence, we can see that the present formulation of results for
infinitely many frequencies is a corollary of 
a-posteriori results for a finite number of frequencies.

Of course, working out the 
strategy requires making precise definitions of sizes for models 
or arbitrary number of variables. A similar strategy was used 
in \cite{FontichLS} to construct almost periodic solutions in 
models of lattice oscillators (the solutions) were localized 
around a sequence of centers.

The most important difference with the method in the previous sections
is that in the construction, we do not need to deal with actually infinite
models and we do not need to solve equations with infinitely many variables.
This leads to Diophantine conditions which are different
from those in Section~\ref{sec: diophantine}.

\subsection{The model considered}

In this section, we will present the details for an 
almost periodic Frenkel-Kontorova model,  
whose equilibrium equations are 
\begin{equation}\label{apFK} 
u_{n+1} - 2 u_n + u_{n-1} + W(u_n) = 0.
\end{equation}

We will assume that $W(t) = \frac{d}{dt} V(t)$
\begin{equation}\label{apW} 
V(t) = \sum_{j \in \mathbb N_+} V_j (\alpha_1 t, \cdots,  \alpha_j t) 
\end{equation}
where the $ \{\alpha_j\}_{j = 1}^\infty\in[0,1]^{\mathbb N}$ are given numbers and 
$V_j: \torus^j_{\rho}  \rightarrow \complex$ are functions 
analytic in  the interior of  $\torus^j_{\rho} $ and 
continuous up to the boundary. Later we will assume decay condition
on $\| V_j \|_{\rho}$, where the norm is given by the supremum.

A physical interpretation of \eqref{apFK} is a sequence of
particles placed at positions $u_n$. The particles experience
harmonic forces among themselves (given by the second difference terms)
and by interaction with the substratum (given by the terms with the $W$).

This is not the only physical interpretation of the model \eqref{apW}.
In \cite{FK39} the original physical interpretation is plane dislocations.
\subsubsection{Some related models} 

An interesting variant model is the Heisenberg XY model of magnetism
\cite{Mattis}. We consider that the particles places at site $i$ have
a spin $\theta_i\in \mathcal{S}^1$ (think of the spin as a unit vector
of angle $\theta_i$). The interaction (physically it is an \emph{exchange}
interaction)
among nearby particles has an energy given by the scalar product of the unit vectors.
Furthermore, each particle is subject to an interaction with a magnetic
field at site $i$ whose strength is $\alpha_i$ and orientation $\beta_i$. 
The equilibrium equations of
Heisenberg XY model are given by
\begin{equation}\label{XY}
  \sin(\theta_{i+1} - \theta_i) - \sin(\theta_i - \theta_{i-1})
  + \alpha_i \cos( \theta_i - \beta_i  ) = 0
\end{equation}

The Frenkel-Kontorova model is a low frequency approximation
($\theta_i = \theta_{i+1}$ small so that
$  \sin(\theta_{i+1} - \theta_i) - \sin(\theta_i - \theta_{i-1})
\approx  (\theta_{i+1} - \theta_i) - (\theta_i - \theta_{i-1})$. 

We note that the physical interpretation in the XY model
makes it very natural to consider variants in which the interaction
has terms that are not just nearest neighbors and which are many body
(these are properties of the exchange interaction).  From the mathematical
point of view, we note that the Frenkel-Kontorova model has variables in $\real$ while
the XY model has variables in the circle.

\begin{rmk} 
We call attention that the paper \cite{LionsS} develops a
homogeneization theory for models of the form \eqref{apFK}. That is,
there is an average energy proportional to the length. On the other hand,
some examples show that, even with just two frequencies, the
corrector (the deviation of the minimizer with respect to
the line) may be unbounded. This is in contrast with the results in
Aubry-Mather theory \cite{AubryLD83} which show that in the one frequency
case, the correctors are bounded and with the results of this paper
that show that, for small potentials, the correctors produced by KAM
theory are bounded.

From the mathematical point of view, it is
interesting to obtain information of the rates of growth of the corrector.

If we consider potentials given by a set of parameters,
for some values of the parameters, the correctors are bounded, but for others
they are not. It is interesting  to understand  what happens at the boundary.
Some preliminary numerical explorations are in \cite{BlassL10}, which
suggest several conjectures (universal scalings at the transition). 
\end{rmk}

\subsection{Diophantine conditions} 

The models for which we can develop a useful theory are the models
for which the frequencies $\{ \alpha_n\}_{n \in \nat}$ satisfy the
following definition.
\begin{defi}
  We say that
  $\{ \alpha_n\}_{n \in \nat} \in \D_h( \{\nu_n\}_{n \in \nat},  \{\tau_n\}_{n \in \nat} ) $
  when for all 
 $N \in \nat$,
$\alpha^{[ \le N]} \equiv (\alpha_1, \ldots , \alpha_N)$  is homogeneous
Diophantine (with constants $\nu_N, \tau_N$) in the standard sense
\cite{Schmidt}.
\begin{equation}\label{Dioph_h}
	\alpha^{[ \le N]} \in \D_{h,N}(\nu_N, \tau_N) \iff  |\alpha^{[\le N]}\cdot k | \ge
  \nu_N |k|_1 ^{-\tau_N}  \quad  \forall k \in \integer^N \setminus \{0\}
\end{equation}
where $|k|_1 = |k_1| + \cdots +|k_N|$.

We denote
\begin{equation}
  \begin{aligned} 
  & \D_{h,N}(\tau_N) = \cup_{ \nu_N } \D_h(\nu_N, \tau_N) \subset \real^N, \\
 &\D_h( \{\tau_n\}_{n\in \nat}) = \cup_{ \{\nu_n\}_{n \in \nat}}\D_h(\{\nu_n\}_{n \in \nat}, \{\tau_n\}_{n \in \nat}).
  \end{aligned}
  \end{equation}
\end{defi}

It is a standard result in Diophantine approximation that $\D_{h,N}(\tau_N)$
has full Lebesgue measure in $\real^N$ provided that $\tau_N > N-1$.
If $\tau_N < N-1$, $\D_{h,N}(\tau_N) = \emptyset$.

From this we can deduce the following density result
in infinite dimensions:
\begin{lem}\label{abundance_h}
    If $\tau_n > n-1$ for all $n\in\mathbb N$, then $$m([0,1]^{\mathbb N}\setminus ( \D_h( \{\tau_n\}_{n\in \nat})  ))=0,$$ where $m(\cdot)$ is the product probability measure.
\end{lem}

\begin{rmk}
  Note that the abundance result Lemma~\ref{abundance_h}
  presented above refers to situations where the $\alpha_n$
  are uniformly bounded. In \cite{Perfetti}, it is remarked that
  if we choose $\alpha_n$ going to infinity sufficiently fast,
  it is very easy to adjust even stronger Diophantine conditions.

  The use of very large frequencies does not seem natural in Frenkel-Kontorova
  type models, but it may be natural in PDE's.
\end{rmk}

Given an $\omega \in \real_+$ we seek solutions of 
\eqref{apFK} of the form:
\begin{equation} \label{hullfunction} 
u_n = \omega n  + \hat h (\omega n \alpha)
\end{equation}
where $\hat h$ is an almost periodic function. That is: 
$\hat h(\theta) = \sum_n  \hat h_n ( \theta \alpha^{[ \le n]} ) $ with
the $\hat h_n : \torus_{\rho}^n\rightarrow \complex$ and
taking real values for real values of the arguments.

We note that if there is a solution of the form \eqref{hullfunction},
we have many of them $ u_n = \omega n  + \hat h (\omega n \alpha + \varphi)$
is also an equilibrium solution for any $\phi$. The fact that there is a continuous of equilibria depending smoothly on parameters has the physical interpretation that the solutions can \emph{slide}. Any small external constant force
applied to all the particles can move them from one equilibrium to another
and get the solutions to move. In the interpretation of particles deposited
on a material, this has the interpretation that the deposited particles do
not stick and that any arbitrarily small tilting force will male them move.

We also note when  there is a continua of equilibria and the
energy functional has some weak convexity conditions (satisfied
by \eqref{apFK}), an argument in \cite{SuL17} -- adapting
to our context ideas of \cite{Caratheodory}-- shows that the
solutions are minimizers in the sense that compactly supported
modifications increase the energy  -- considering only the terms
which change --(class A minimizers).

Note that the sequence of frequencies  $\{ \alpha_n \}_{n \in \nat}$
is a property of the model. They are part of the description of
the substratum. The unknowns
are $\omega$ which has the interpretation of the inverse of
an average density (notice  that since $\hat h$ is bounded,
we have that, for large $n$, we have $u_n \approx \omega n$.

A key assumption of the result Theorem~\ref{new} is is
the $\omega  \alpha$ satisfies a non-homogeneous  Diophantine condition.
If we consider $\alpha$ fixed (the $\alpha$ is a property of the material)
it becomes a condition on the $\omega$, the densities admitting sliding.

\begin{defi}
  We say that
  $\{ \omega  \alpha_n\}_{n \in \nat} \in \D( \{\nu_n\}_{n \in \nat},  \{\tau_n\}_{n \in \nat} ) $
  when for all 
 $N \in \nat$,
$\alpha^{[ \le N]} \equiv (\alpha_1, \ldots , \alpha_N)$  is non-homogenous
Diophantine (with constants $\nu_N, \tau_N$) in the standard sense
\cite{Schmidt}.
\begin{equation}\label{Dioph}
\omega	\alpha^{[ \le N]} \in \D_N(\nu_N, \tau_N) \iff  |\omega \alpha^{[\le N]}\cdot k - 2\pi n | \geq
  \nu_N |k|_1 ^{-\tau_N}  \quad  \forall k \in \integer^N \setminus \{0\}\,\,\forall n\in\mathbb Z
\end{equation}
where $|k|_1 = |k_1| + \cdots + |k_N|$.

We denote $\D(\{\tau_n\}_{n\in\mathbb N})=\cup_{\{\nu_n\}_{n\in\mathbb N}}\D(\{\nu_n\}_{n\in\mathbb N},\{\tau_n\}_{n\in\mathbb N})$.
\end{defi}

\begin{lem}
Assume that $\alpha \in [0,1]^{\mathbb N}\cap \D_h( \{\tau_n\})$ and 

	Let $\tilde \tau_n > \tau_n + n$. 
	Then, the set of $\omega\in \real$ for which 
	$\omega \alpha \in \D(\{ \tilde \tau_n\}) $ has 
	full measure. 

\end{lem} 

\begin{proof}
Denote
$$\tilde\D(\{\tilde{\tau}_n\};\{\alpha_n\})=\{\omega\in\mathbb R:\{\omega\alpha_n\}\in\D(\{\tilde\tau_n\})\},$$
$$\tilde\D(\{\nu_n\},\{\tilde{\tau}_n\};\{\alpha_n\})=\{\omega\in\mathbb R:\{\omega\alpha_n\}\in\D(\{\nu_n\},\{\tilde\tau_n\})\}.$$          
Let $A>0$. Consider the sets $$\mathscr{B}_{N,k,n}^{\{\nu_s\}}=\{\omega:|\omega\alpha^{[\le N]}\cdot k-2\pi n|<\nu_N|k|^{-\tilde{\tau}_N}\}$$
and $$\mathscr{B}_{N,k,n,A}^{\{\nu_s\}}=\mathscr{B}_{N,k,n}^{\{\nu_s\}}\cap [A,1.01A].$$
Clearly, 
\begin{equation*}
    \begin{aligned}
        \relax [A,1.01A]\setminus\tilde\D(\{\tilde{\tau}_n\};\{\alpha_n\}) &= \cap_{\{\nu_n\}}([A,1.01A]\setminus\tilde{\D}(\{\nu_n\},\{\tilde{\tau}_n\};\{\alpha_n\})) \\
        & = \cap_{\{\nu_s\}}\cup_{N\in\mathbb N}\cup_{k\in\mathbb Z^N\setminus\{0\},n\in\mathbb Z}\mathscr B_{N,k,n,A}^{\{\nu_s\}}
    \end{aligned}
\end{equation*}
Note that $\mathscr B_{N,k,n}^{\{\nu_s\}}$ is an interval of length $$|\mathscr B_{N,k,n}^{\{\nu_s\}}|=\frac{2\nu_N|k|^{-\tilde{\tau}_N}}{|\alpha^{[\le N]}\cdot k|},$$
and $$\sharp\{n:\mathscr B_{k,n,A}^{\{\nu_s\}}\ne\varnothing\}\le\frac{0.02A}{2\pi}{|\alpha\cdot k|}+2.$$
Since $\alpha\in\D_h(\{\tau_n\})$, there exists $\{\nu'_n\}_{n\in\mathbb N}$ such that $\alpha\in\D_h(\{\nu_n'\},\{\tau_n\})$.
Hence
\begin{equation*}
    \begin{aligned}
        |\cup_{N,k,n}\mathscr B_{N,k,n,A}^{\{\nu_s\}}|&\le\sum_{\substack{N,k,n\\ \mathscr B^{\{\nu_s\}}_{N,k,n,A}\ne\varnothing}}\frac{2\nu_N|k|^{-\tilde{\tau}_N}}{|\alpha^{[\le N]}\cdot k|} \le \sum_{N,k} \frac{2\nu_N|k|^{-\tilde\tau_n}}{|\alpha^{[\le N]}\cdot k|}(\frac{0.02A}{2\pi}|\alpha^{[\le N]}\cdot k|+2)\\
        &\le\sum_N(\frac{0.04A}{2\pi} A\cdot\nu_N\sum_{k\in\mathbb Z^N\setminus\{0\}}|k|^{-\tilde\tau_N}+\frac{4\nu_N}{\nu'_N}\sum_{k\in\mathbb Z^N\setminus\{0\}}|k|^{-(\tilde\tau_N-\tau_N)})\\
        &\le\sum_N(\frac{0.04A}{2\pi} A\cdot\nu_N C_1(\tilde\tau_N)+\frac{4\nu_N}{\nu'_N}C_2(\tilde\tau_N,\tau_N)).
    \end{aligned}
\end{equation*}
By appropriately selecting $\{\nu_s\}_{s\in\mathbb N}$, $|\cap_{\{\nu_s\}}\cup_{N,k,n}\mathscr B^{\{\nu_s\}}_{N,k,n,A}|$ can be arbitrarily small. Therefore, $\tilde\D(\{\tilde{\tau}_n\};\{\alpha_n\})$ is of full measure.

\end{proof}


\begin{rmk}
  It is interesting to compare the Diophantine conditions 
  in this section with the infinite dimensional
  Diophantine conditions in Section~\ref{sec: diophantine}.

  In this section, we assume that for each $N$,  $\alpha^{[\le N]} $ are Diophantine in the
  standard  sense with the resonance bounded that are a power of
  the $\ell_1$-norm of $k$. This allows Diophantine conditions $\tau_N > N-1$
  and allow, in principle any decrease of the $\nu_N$.  Of course, if the
  $\nu_N$ decrease very fast, we need stronger conditions on the decrease
  of $\| W_N \|_{\rho_n}$. 

  In  \eqref{dio1}, the conditions involve other weights, which are
  powers of the components of the $k$ with weights. In terms of the
  usual conditions, this leads to $\tau_N > (1 + \tau)N$ and with $\nu_N$ that
  decrease like $N!^{1+\tau}$. Indeed, the small denominator
  bounds have more structure. By fixing these rates of growth of the conditions
  and taking advantage of the structures, the
  method here
  allows to use rather explicit
  sufficient conditions on $\rho_N$,
  $\| W_N \|_{\rho_N}$.

  The conditions in ~\eqref{dio1} allow to consider cohomology
  equations for functions in infinitely many variables. In the method of
  this section, we only use the standard
  finite dimensional cohomology equations.
  \end{rmk}

  \subsection{Statement of results}

  The main result of this section  is the following:

 \begin{thm}\label{new}
   Consider the model \eqref{apFK}. Denote $W_j(t) = \frac{d}{dt} V_j (\alpha_1 t, \cdots, \alpha_j t )$ and $\widehat W_j$ is the shell function of $W_j$, $\widehat W_j(\alpha^{[\le j]}t) = W_j(t)$. Let $\omega \in \real $ such that
    $\omega \alpha \in \D_h( \{\nu_n\}, \{ \tau_n\} )$. 
    
    Let $0<\rho_{\infty}<\rho$. Assume that the $W_n$ are analytic on $\torus^n_{\rho}$ for all $n\in\nat_+$,  and
    \begin{equation}\label{recursive} 
      \| \widehat W_n \|_{C^2(\mathbb T_{\rho}^n)} \le C_n( \{ \nu_j \}_{j\le n}, \{ \tau_j \}_{j\le n}, \{ \| \widehat W_j \|_{C^2(\mathbb T_{\rho}^j)} \}_{j\le{n-1}} ), \,\, \forall n\in\mathbb N_+,
    \end{equation} 
    where the $C_n > 0$ are given in the proof. Moreover, $\sum_{n}\| \widehat W_n \|_{C^2(\mathbb T_{\rho}^n)}<+\infty$.

	 Then, we can find an almost periodic hull 
	 function $h(t) = t + \hat h (\alpha t) $, such that $u_n = h(\omega n)$ is solution of \eqref{apFK}, and $\hat h$ is analytic on $\mathbb T_{\rho_{\infty}}^{\mathbb N}$, where $0<\rho_{\infty}<\rho$. 
 \end{thm} 

 \begin{rmk}\label{analytic}
     Here, by analytic on the infinite deminsional torus $\mathbb T^{\mathbb N}_{\rho_{\infty}}$, we mean $\hat h(\sigma)=\sum_{n\in\mathbb N} \widehat H_n(\sigma^1,\cdots,\sigma^n )$, where $\widehat H_n(\sigma^1,\cdots,\sigma^n)$ are analytic in the usual sense on the $\mathbb T_{\rho_{\infty}}^n$, and $\sum_{n\in\mathbb N}\|\widehat H_n\|_{\rho_{\infty}}<\infty$.
 \end{rmk}

 \subsection{Proof of Theorem~\ref{new}}

 We will deduce Theorem~\ref{new} from Theorem 12 of
 \cite{ref2}.

 We present now a simplified version.
 \begin{thm}\label{oldtheorem}
     Let $h = Id + \tilde h$, where $\tilde h(\theta) = \hat h (\alpha \theta) = \sum_{k \in \mathbb Z^d}\hat h_k \cdot e^{ik\cdot \alpha \theta}$ with $\hat h_0 = 0$, $\hat h\in\mathscr A _{\rho_0}^1$, and $\alpha\in\mathbb R^d$ such that $\alpha\cdot j \ne 0$, $\forall j \in \mathbb Z^d \setminus \{ 0 \}$. Denote $\hat l=1+\partial_\alpha\hat h$, $T_{x}(\sigma)=\sigma+x$, We assume the following: 
\begin{itemize}
\item[{\rm (H1)}]
Diophantine condition: $|\omega\alpha\cdot k-2n\pi|\ge\nu |k|^{-\tau}$ holds for $\forall k\in\mathbb Z^d\setminus\{0\}$, $\forall n\in\mathbb Z$, where $\tau,\nu>0$.
\item[{\rm (H2)}]
Non-degeneracy condition: $\|\hat l(\sigma)\|_{\rho_0}\le N^+,\ \|(\hat l(\sigma))^{-1}\|_{\rho_0}\le N^-$, $\big|\big<\frac1{\hat l\cdot\hat l\circ T_{-\omega\alpha}}\big>\big|\ge c>0$. 
\item[{\rm (H3)}]
Composition condition: Take $\iota:={\rm dist}(\mathbb C^d\setminus\Omega,(Id+\alpha\hat h)(\mathbb T_{\rho_0}^{d}))$, where $\Omega$ is the domain of $\widehat U$. We assume that $\|\hat h\|_{\rho_0}+\rho_0\le\frac12\iota$. 
\end{itemize}

Let $\rho_{\infty}=\rho_0-\delta$, where $0<\delta<\rho_0/2$. 
There exists a constant 
\[\epsilon^*=\epsilon^*(N^-,N^+,d,\tau,c,\iota,\|U\|_{C^2(\Omega)})>0,\]
 such that, if $\|\mathcal E[\hat h,\lambda]\|_{\rho_0}\le \epsilon^* \nu^4 \delta^{4\tau}$, then there exists a periodic function $\hat h^*$ and $\lambda^*\in\mathbb R$, such that 
$$\mathcal E[\hat h^*,\lambda^*]=0.$$ Moverover, there exists a constant $C=C(N^-,N^+,d,\tau,c,\iota,\|U\|_{C^2(\Omega)})>0$, such that $$\| \hat h - \hat h^* \|_{\rho_{\infty}}\le C\nu^{-2}\delta^{-2\tau}\|\mathcal{E}[\hat h,\lambda]\|_{\rho_0},$$
$$|\lambda-\lambda^*| \le C\|\mathcal{E}[\hat h,\lambda]\|_{\rho_0}.$$
Denote $\hat l^*=1+\partial_{\alpha}\hat h^*$, we have: 
$$\|\hat l^* (\sigma)\|_{\rho_{\infty}}\le N^+ + C\nu^{-2}\delta^{-2\tau-1}\|\mathcal{E}[\hat h,\lambda]\|_{\rho_0},$$
$$\|(\hat l^* (\sigma))^{-1}\|_{\rho_{\infty}}\le N^- + C\nu^{-2}\delta^{-2\tau-1}\|\mathcal{E}[\hat h,\lambda]\|_{\rho_0},$$
$$\big|\big<\frac1{\hat l^*\cdot\hat l^*\circ T_{-\omega\alpha}}\big>\big| \ge c - C\nu^{-2}\delta^{-2\tau-1}\|\mathcal{E}[\hat h,\lambda]\|_{\rho_0}.$$

$${\rm dist}(\mathbb C^d\setminus\Omega,(Id+\alpha\hat h)(\mathbb T_{\rho_0}^d))\ge\iota-C\nu^{-2}\delta^{-2\tau}\| 
\E[\hat h,\lambda] \|_{\rho_0}.$$

 \end{thm}

 Note that Theorem~\ref{new} is in an a-posteriori format.
 Given an approximate solution that satisfies some  non-degeneracy conditions,
 then we get a true solution, and moreover, we get that the exact solution is
 close to the approximate solution. The distance between the approximate
 solution and the true solution can be bounded -- in a slightly smaller
 domain --  by the size of the residual of the approximate solution. 

  In is important to observe that  the non-degeneracy conditions of
 Theorem~\ref{oldtheorem} are rather explicit expressions obtained by averaging
 quantities related to the approximate solution. In particular,
 the difference between the non-degeneracy conditions of the true
 solutions and the non-degeneracy conditions of the true solution are
 bounded by the error in the initial approximation.

\cite[Theorem 12]{ref2} contains  some generalizations.
 It does not assume that $W_n$ are derivatives, but adds an extra constant force. After the work is done, it is shown that, when the $W_n$ are derivatives,
 the extra parameter vanishes. 

 Take $\rho_n=\rho_{\infty}+2^{-n-1}(\rho-\rho_{\infty})$, and $\delta_{n+1}=\rho_n-\rho_{n+1}$. Denote $\hat l^n = 1 + \partial_{\alpha^{[\le n]}} \hat h^n$.
 
 To deduce Theorem~\ref{new} from Theorem~\ref{oldtheorem},
 we consider a sequence of models. For all $n\in\mathbb N_+$, $\E_n$ acting on $\hat h:\mathbb T_{\rho_n}^n \to \mathbb C$, we have
 \[
   \E_n[ \hat h](\sigma_n) \equiv \hat h_n (\sigma_n + \omega\alpha^{[\le n]}) + \hat h(\sigma_n-\omega\alpha^{[\le n]}) + 2\hat h(\sigma_n) + \sum_{j\le n} \widehat W_j(\sigma_n^{[\le j]} +\alpha^{[\le j]}\hat h(\sigma_n)),
 \]
 where $\sigma_n=(\sigma^1,\cdots,\sigma^n)\in\mathbb T_{\rho_n}^n$.
 Let $\hat h^n:\mathbb T_{\rho_n}^n \to \mathbb C$. 
 
 We observe that, provided that the compositions can be defined, 
  \[
   \E_{n+1} [\hat h^n](\sigma_{n+1}) =  \E_n[\hat h^n](\sigma_n) + \widehat W_{n+1}(\sigma_{n+1}+\alpha^{[\le n+1]}\hat h^n(\sigma_n)).
  \]
 
 
 When $n=1$, taking $\hat h^0 = 0$, we have $\E_1[\hat h^0](\sigma_1)=\widehat W_1(\sigma_1)$ and 
 $\| \hat l^0 \|_{\rho_0} = 1 $, 
 $\| (\hat l^0)^{-1} \|_{\rho_1} = 1 $, 
 $\big| \big<\frac1{\hat l^0\cdot\hat l^0\circ T_{-\omega\alpha^{[\le 1]}}}\big> \big| = 1$, 
 ${\rm dist}(\mathbb C^1 \setminus \mathbb T_{\rho}^1, \mathbb T_{\rho_0}^1 ) = \frac{\rho-\rho_{\infty}}{2}$. 
 By Theorem \ref{oldtheorem}, if $\| \widehat W_1 \|_{\rho_0} \le \epsilon_1^* \nu_1^4 \delta_1^{4\tau_1}$, where $\epsilon_1^*=\epsilon^*(2,2,\frac{1}{2},1,\tau_1,\frac{\rho-\rho_{\infty}}{4},\|\widehat W_1\|_{C^2(\mathbb T_{\rho_0}^1)})$, 
 there exists $\hat h^1 : \mathbb T_{\rho_1}^1 \to \mathbb C$ such that $\E [\hat h^1]=0$ and $$\|\hat h^1\|_{\rho_1} \le C_1 \nu_1^{-2} \delta_1^{-2\tau_1} \| \widehat W_1 \|_{\rho_0} , $$ 
 where $C_1=C(2,2,\frac{1}{2},1,\tau_1,\frac{\rho-\rho_{\infty}}{4},\|\widehat W_1\|_{C^2(\mathbb T_{\rho_0}^1)})$ is given by Theorem \ref{oldtheorem}. 
 Moreover, if $$C_1 \nu_1^{-2}\delta_1^{-2\tau_1-1}\|\widehat W_1\|_{\rho_0} \le \frac{1}{2^2},$$
 $$C_1\nu_1^{-2}\delta_1^{-2\tau_1}\|\widehat W_1\|_{\rho_0}\le \frac{\rho-\rho_{\infty}}{2^3},$$
 then 
 $$\| \hat l^1 \|_{\rho_1} \le 1 + \frac{1}{2^2} \le 2,$$
 $$\| (\hat l^1)^{-1} \|_{\rho_1} \le 1 + \frac{1}{2^2} \le 2,$$
 $$\big| \big<\frac1{\hat l^1\cdot\hat l^1\circ T_{-\omega\alpha^{[\le 1]}}}\big> \big| \ge 1 - \frac{1}{2^2} \ge \frac12,$$
 $${\rm dist}(\mathbb C^1 \setminus \mathbb T_{\rho}^1, (Id + \alpha^{[\le1]}\hat h^1)(\mathbb T_{\rho_1}^1) ) \ge (\frac{1}{2}-\frac{1}{2^3})(\rho-\rho_{\infty}) \ge \frac{1}{4}(\rho-\rho_{\infty}) . $$

  
 For general $n$, if $\E_n[\hat h^n]=0$, $\hat h^n:\mathbb T_{\rho_n}^n \to \mathbb C$, we have 
 \[\E_{n+1}[\hat h^n](\sigma_{n+1})=\widehat W_{n+1}(\sigma_{n+1}+\alpha^{[\le n+1]}\hat h^n(\sigma_n)).\]

 Furthermore, if $\hat h^n$ satisfies the non-degeneracy condition and composition condition, by Theorem \ref{oldtheorem}, 
 when 
 \[
 \begin{split}
  \| \widehat W_{n+1} \|_{\rho_n} &\le \epsilon_{n+1}^*\nu_{n+1}^4\delta_{n+1}^{4\tau},\\ 
 \epsilon_{n+1}^*&=\epsilon^*(2,2,\frac{1}{2},n+1,\tau_{n+1},\frac{\rho-\rho_{\infty}}{4},\| \sum_{j\le n+1}\widehat W_j\|_{C^2(\mathbb T_{\rho_{n}}^{n+1})}),
 \end{split}
 \] 
 we can obtain $\hat h^{n+1}: \mathbb T_{\rho_{n+1}}^{n+1} \to \mathbb C$ such that $\E_{n+1}[\hat h^{n+1}]=0$ and 
 $$ \| \hat h^{n+1} - \hat h^n \|_{\rho_{n+1}} \le C_{n+1} \nu_{n+1}^{-2} \delta_{n+1}^{-2\tau_{n+1}} \| \widehat W_{n+1} \|_{\rho_{n}}, $$ 
 where $C_{n+1}=C(2,2,\frac{1}{2},n+1,\tau_{n+1},\frac{\rho-\rho_{\infty}}{4},\| \sum_{j\le n+1}\widehat W_j\|_{C^2(\mathbb T_{\rho_{n}}^{n+1})})$. 
 
 Moreover, assuming that 
 $$ C_{n+1} \nu_{n+1}^{-2} \delta_{n+1}^{-2\tau_{n+1}-1} \| \widehat W_{n+1} \|_{\rho_{n}} \le \frac{1}{2^{n+1}}, $$
 $$ C_{n+1} \nu_{n+1}^{-2} \delta_{n+1}^{-2\tau_{n+1}} \| \widehat W_{n+1} \|_{\rho_{n}} \le \frac{\rho-\rho_{\infty}}{2^{n+2}} , $$
 then 
 $$ \| \hat h^{n+1} - \hat h^n \|_{\rho_{n+1}} \le \frac{\rho-\rho_{\infty}}{2^{n+2}},$$
 $$ \| \hat l^{n+1} \|_{\rho_{n+1}} \le \| \hat l^n \|_{\rho_n} + \frac{1}{2^{n+1}}, $$
 $$ \| (\hat l^{n+1})^{-1} \|_{\rho_{n+1}} \le \| (\hat l^{n})^{-1} \|_{\rho_{n}} + \frac{1}{2^{n+1}},$$
 $$ \big| \big<\frac1{\hat l^{n+1}\cdot\hat l^{n+1}\circ T_{-\omega\alpha^{[\le n+1]}}}\big> \big| \ge \big| \big<\frac1{\hat l^n\cdot\hat l^n\circ T_{-\omega\alpha^{[\le n]}}}\big> \big| - \frac{1}{2^{n+1}}, $$
 $$ {\rm dist}(\mathbb C^{n+1} \setminus \mathbb T_{\rho}^{n+1}, (Id + \alpha^{[\le n+1]}\hat h^{n+1})(\mathbb T_{\rho_{n+1}}^{n+1}) ) \ge {\rm dist}(\mathbb C^n \setminus \mathbb T_{\rho}^n, (Id + \alpha^{[\le n]}\hat h^n)(\mathbb T_{\rho_n}^n) ) - \frac{\rho-\rho_{\infty}}{2^{n+2}}. $$

 The issue now is to study the limit of the sequence $\{\hat h^n\}_{n\ge1}$. Actually, the following two lemmas guarantee that the limit function $\hat h$ is the solution in Theorem \ref{new}.

 \begin{lem}
     $\{\hat h^n\}_{n\in\mathbb N}$ converges uniformly on every compact set of $\mathbb T^{\mathbb N}$. Denote by $\hat h$ the limit obtained in this sense, then $\hat h$ is analytic in $\mathbb T_{\rho_{\infty}}^{\mathbb N}$.
 \end{lem}

 \begin{proof}
     For all $\sigma \in \mathbb T^{\mathbb N}$, we have
     $$\lim_{n\to\infty} \hat h^n (\sigma) = \hat h^0 (\sigma_1)+\sum_{n\in\mathbb N} ( \hat h^{n+1}(\sigma_{n+1})  - \hat h^n (\sigma_n) ).$$
     Note that 
     \[
     \begin{aligned}
         \| \sum_{n\in\mathbb N} (\hat h^{n+1}(\sigma_{n+1})-\hat h^n(\sigma_n)) \|_{\rho_{\infty}} & \le \sum_{n\in\mathbb N} \| \hat h^{n+1} - \hat h^{n} \|_{\rho_{\infty}} \\
         & \le \sum_{n\in\mathbb N} \| \hat h^{n+1} - \hat h^n \|_{\rho_{n+1}} \le \sum_{n\in\mathbb N} \frac{\rho-\rho_{\infty}}{2^{n+2}}.
     \end{aligned}
     \]
     Therefore $\{\hat h^n\}_{n\in\mathbb N}$ converges uniformly on compact sets of $\mathbb T^{\mathbb N}$ and the limit function $\hat h$ is analytic on $\mathbb T_{\rho_{\infty}}^{\mathbb N}$ in the sense of Remark \ref{analytic}.
 \end{proof}

 \begin{lem}
     The limit function $\hat h$ satisfies $\E_{\infty}[\hat h](\sigma) = \hat h(\sigma+\omega\alpha) + \hat h(\sigma-\omega\alpha) - 2\hat h(\sigma) +\sum_{j} \widehat W_j(\sigma_j+\alpha^{[\le j]}\hat h(\sigma)) = 0$.
 \end{lem}

 \begin{proof}
     For every $n\in\mathbb N$, we have $\E_n[\hat h^n]=0$.
     Then 
     \[
     \begin{aligned}
         \| \E_{\infty}[\hat h](\sigma) \|_{\rho_{\infty}}=&\| \hat h(\sigma + \omega\alpha) + \hat h(\sigma - \omega\alpha) - 2\hat h(\sigma) + \sum_{j} \widehat W_j(\sigma_j+\alpha^{[\le j]}\hat h(\sigma))\|_{\rho_{\infty}}\\
         =& \| \hat h(\sigma + \omega\alpha) + \hat h(\sigma - \omega\alpha) - 2\hat h(\sigma) + \sum_{j} \widehat W_j(\sigma_j+\alpha^{[\le j]}\hat h(\sigma)) \\
         &- \hat h^n (\sigma + \omega \alpha) - \hat h^n (\sigma-\omega\alpha) + \hat h^n(\sigma) - \sum_{j\le n} \widehat W_j(\sigma_j+\alpha^{[\le j]}\hat h^n(\sigma)) \|_{\rho_{\infty}}\\
         \le& 4\| \hat h - \hat h^n \|_{\rho_{\infty}} + \sum_{j \le n} \| \widehat W_j \|_{C^1(\mathbb T_{\rho_{\infty}}^{j})}\| \hat h^n - \hat h \|_{\rho_{\infty}} +\sum_{j \ge n+1}\| \widehat W_j \|_{\rho_{\infty}}\\
         \le& 4\| \hat h - \hat h^n \|_{\rho_{\infty}} + \sum_{j}\| \widehat W_j \|_{C^1(\mathbb T_{\rho_{\infty}}^{j})}\| \hat h^n - \hat h \|_{\rho_{\infty}} + \sum_{j\ge n+1} \| \widehat W_j \|_{\rho_{\infty}}.
     \end{aligned}
     \]
     Since $\hat h^n$ converges to $\hat h$ and $\sum_{j}\| \widehat W_j \|_{C^2(\mathbb T_{\rho}^{j})}$ is bounded, we obtain the desired result.
 \end{proof}

 \begin{rmk}
   The paper \cite{ref2} also contains a version for Sobolev
   regularity. There does not seem to a similar version for  the
   almost periodic case.   Each step of the Sobolev result entails a
   loss of regularity related to the Diophantine exponent.
   Since we need to perform infinitely many steps and each step will
   have an increasing loss of regularity in the Sobolev scale, we can
   only perform a finite number of steps.
   \end{rmk}

 \section{Long-range model with almost-periodic potential}
\label{S5}
\subsection{Long-range models}
Similar to the short range model, the configuration of the system is described by $u=\{u_n\}_{n\in\mathbb{Z}}$ with $u_n\in\mathbb{R}$ in the long range model. The interaction is obtained by assigning an energy to every finite subset of $\mathbb Z$. We consider the formal energy functional \eqref{long}.

Take the formal derivative of $\mathscr S(u)$ with respect to $u_j$ and set it to zero:
\begin{equation}\label{y3}
\frac{\partial\mathscr S}{\partial u_j}=\sum_{L=0}^\infty\sum_{i=j-L}^j\alpha\cdot\partial_{j-i}\widehat H_L(u_i\alpha,\cdots,u_{i+L}\alpha)=0,
\end{equation}
where $\partial_j=\frac{\partial}{\partial\sigma_j}$ for $j=0,1,\cdots,L$. We denote $\partial_\alpha^{(j)}=\alpha\cdot\partial_j$ in the following. It is clear that the operators $\partial_\alpha^{(j)}$ are commutative, \text{i.e.} $\partial_\alpha^{(j)}\partial_\alpha^{(k)}=\partial_\alpha^{(k)}\partial_\alpha^{(j)}$, for any $j,k\in\mathbb{Z}$.

\subsection{Hull function}
We also write $u_n$ of the form:
\begin{equation}\label{y1}
u_n=h(n\omega)=n\omega+\hat h(n\omega\alpha),
\end{equation}
where $\hat h$ is a function on $\mathbb{T^N}$ and $n\in\mathbb{Z}$, $\omega\in\mathbb{R}$, $\alpha\in[0,1]^\mathbb{N}$.

Denote $$\theta=n\omega,\sigma=\theta\alpha,$$
\begin{equation}\label{y2}
\begin{aligned}
h^{(j)}(\theta)&=h(\theta+j\omega)=\theta+j\omega+\hat h(\sigma+j\omega\alpha),\\
\gamma^{(j)}_L(\theta)&=(h^{(j)}(\theta)\alpha,\cdots,h^{(j+L)}(\theta)\alpha).
\end{aligned}
\end{equation}
In particular, we denote $h(\theta)=h^{(0)}(\theta)$, $\gamma_L(\sigma)=\gamma^{(0)}_L(\sigma)$.\smallskip

Using the notations (\ref{y1}) and (\ref{y2}), we can write (\ref{y3}) more concisely as
\begin{equation*}
\sum_{L=0}^\infty\sum_{i=j-L}^j\partial_\alpha^{(j-i)}\widehat H_L(\gamma_L^{(i)}(\theta))=0,\quad\forall j\in\mathbb{Z}.
\end{equation*}
That is,
\begin{equation}\label{y4}
\sum_{L=0}^\infty\sum_{k=0}^L\partial_\alpha^{(k)}\widehat H_L(\gamma_L^{(j-k)}(\theta))=0,\quad\forall j\in\mathbb{Z}.
\end{equation}

If $\omega$ satisfies the Diophantine property, (\ref{y4}) holds if and only if $\mathscr E[\hat h](\theta)$ defined below vanishes identically:
\begin{equation}\label{y5}
\begin{aligned}
\mathscr E[\hat h](\theta)&\equiv\sum_{L=0}^\infty\sum_{k=0}^L\partial_\alpha^{(k)}\widehat H_L(\gamma_L^{(-k)}(\theta))\\
&\equiv\sum_{L=0}^\infty\sum_{k=0}^L\partial_\alpha^{(k)}\widehat H_L(h(\theta-k\omega)\alpha,\cdots,h(\theta-k\omega+L\omega)\alpha)=0,\quad\forall\theta\in\mathbb R.
\end{aligned}
\end{equation}

\subsection{The main theorem of long range models}
\begin{thm}[Long-range KAM theorem]\label{lr}
Let $h(\theta)=\theta+\tilde h(\theta)$, $\tilde h(\theta)=\hat h(\alpha\theta)=\sum_{k\in\mathbb Z_*^{\mathbb N}}\hat h_k e^{i k\cdot\alpha\theta}$, with $\hat h_0=0, \hat h\in\mathscr A_{\rho_0}^1$. $\alpha\in[0,1]^{\mathbb N}$ is rationally independent. Denote $\hat l=1+\partial_\alpha\hat h$ and $T_{-\omega\alpha}(\sigma)=\sigma-\omega\alpha$. We assume:
 \begin{itemize}
\item[{\rm (H1)}]
 Diophantine condition: for some $\tau,\nu>0$ $$|\omega\alpha\cdot k-2n\pi|\ge\frac\nu{\prod_{j\in\mathbb N}(1+\langle\langle j\rangle\rangle^{1+\tau}|k_j|^{1+\tau})},\,\forall k\in\mathbb Z^\mathbb{N}_*\setminus\{0\},\,\forall n\in\mathbb Z.$$
\item[{\rm (H2)}]
 Non-degeneracy condition:$$\|\hat l(\sigma)\|_{\rho_0}\le N^+,\, \|(\hat l(\sigma))^{-1}\|_{\rho_0}\le N^-, \,|\langle\frac1{\hat l\cdot\hat l\circ T_{-\omega\alpha}}\rangle|\ge c>0.$$

\item[{\rm(H3)}]
The interactions $H_L\in\mathcal H_{L,\rho_0+\|\hat h\|_{\rho_0}+\iota}^3$ for some $\iota>0$. Denote
\[
M_L=\max_{i=0,1,2,3}(\|D^iH_L\|_{L, \rho_0+\|\hat h\|_{\rho_0}+\iota})
\]
$$\beta=C\sum_{L\ge2}M_LL^4$$
where C is a combinatorial constant that will be made explicit during the proof.
\item[{\rm(H4)}]
Assume that the inverses indicated below exist and have the indicated bounds:
$$\|(\partial_{\alpha}^{(0)}\partial_{\alpha}^{(1)}\widehat H_1)^{-1}\|_{1,\rho_0+\|\hat h\|_{\rho_0}+\iota}\le T,$$
$$\left |\left (\int_{\mathbb T^{\mathbb N}}\mathcal C_{0,1,1}^{-1}\right)^{-1}\right|\le U$$
where $\mathcal C_{0,1,1}$ is defined in (\ref{y21})\smallskip
 \item[{\rm(H5)}]
$(N^-)^2T\beta<\frac12, (N^-)^2UT<\frac12.$
\end{itemize}
Assume furthermore that $\|\mathscr E[\hat h]\|_{\rho_0}\le\epsilon\le\epsilon^*(N^+, N^-, c, \tau, \nu,\iota, \rho_0, M_L, \beta, T, U)$ where $\epsilon^*>0$ is a function which we will make explicit along the proof.

Then, there exists an analytic function $\hat h^*\in\mathscr A_{\frac{\rho_0}{2}}$ such that $\mathscr E[\hat h^*]=0$.
Moreover,$$\|\hat h-\hat h^*\|_\frac{\rho_0}2\le C_1(N^+, N^-, c, \tau, \nu, \iota, \rho_0, M_L, \beta, T, U)\epsilon_0.$$
The solution $\hat h^*$ is the only solution of $\mathcal{E}[\hat h^*]=0$ with zero average for $\hat h^*$ in a ball centered at $\hat h$ in $\mathscr A_{\frac{3}{8}\rho_0}$, i.e. $\hat h^*$ is the unique solution in the set
$$\{\hat g\in\mathscr A_{\frac{3}{8}\rho_0}: \langle \hat g\rangle=0, \|\hat g-\hat h\|_{\frac{3}{8}\rho_0}\le C_2(N^+, N^-, c, \tau, \nu, \iota, \rho_0, M_L, \beta, T, U)\}$$

where $C_1$, $C_2$ will be made explicit along the proof.
\end{thm}

\subsection{Proof of Theorem~\ref{lr}}
\subsubsection{Quasi-Newton iteration}
We will use the iterative procedure by modifying the Newton method that gives an approximate solution $\hat h$ of (\ref{y5}). First we should find a solution of 
\begin{equation}\label{y6}
D\mathscr E[\hat h]\cdot\widehat\Delta=-\mathscr E[\hat h],
\end{equation}
where $D$ denotes the derivative of the functional $\mathscr E$ with respect to $\hat h$. Then $\hat h+\widehat\Delta$ will be a better approximate solution of (\ref{y5}).

We compute that
\begin{equation*}
(D\mathscr E[\hat h]\cdot\widehat\Delta)(\theta)=\sum_{L=0}^\infty\sum_{k=0}^L\sum_{j=0}^L\partial_\alpha^{(k)}\partial_\alpha^{(j)}\widehat H_L(\gamma_L^{(-k)}(\theta))\widehat\Delta(\sigma-k\omega\alpha+j\omega\alpha).
\end{equation*}
A direct calculation implies:
\begin{equation*}
\frac{d}{d\theta}\mathscr E[\hat h](\theta)=\sum_{L=0}^\infty\sum_{k=0}^L\sum_{j=0}^L\partial_\alpha^{(k)}\partial_\alpha^{(j)}\widehat H_L(\gamma_L^{(-k)}(\theta))(1+\partial_\alpha\hat h(\sigma-k\omega\alpha+j\omega\alpha)).
\end{equation*}

Let $\hat l(\sigma)=1+\partial_\alpha\hat h(\sigma)$ then we obtain the important identity of $D\mathscr E[\hat h]$:
\begin{equation}\label{y8}
\frac{d}{d\theta}\mathscr E[\hat h](\theta)=(D\mathscr E[\hat h]\cdot\hat l)(\theta).
\end{equation}

Unfortunately, equation (\ref{y6}) is hard to solve and we will modify it into the following equation: 
\begin{equation}\label{y7}
\hat l(D\mathscr E[\hat h]\cdot\widehat\Delta)-\widehat\Delta(D\mathscr E[\hat h]\cdot\hat l)=-\hat l\mathscr E[\hat h].
\end{equation}
Equation (\ref{y7}) is just equation (\ref{y6}) multiplied by $\hat l$ and added the extra term in $\widehat\Delta(D\mathscr E[\hat h]\hat l)$ on the left-hand side. 

Due to (\ref{y8}), one can write
\begin{equation}\label{y16}
\widehat\Delta(D\mathscr E[\hat h]\cdot\hat l)=\widehat\Delta\frac{d}{d\theta}\mathscr E[\hat h].
\end{equation}
The reason why this term is small and it does not affect the quadratic character of the procedure will be discussed in Section \ref{5.5}.

Let 
\begin{equation}\label{y9}
\widehat\Delta=\hat l\cdot\hat\eta.
\end{equation} The unknowns $\widehat\Delta$ and $\hat\eta$ are equivalent due to the non-degeneracy assumption in Theorem~\ref{lr}

Substituting (\ref{y9}) into (\ref{y7}), we obtain 
\begin{equation}\label{y10}
\begin{split}
\sum_{L=0}^\infty\sum_{k=0}^L&\sum_{j=0}^L\partial_\alpha^{(k)}\partial_\alpha^{(j)}\widehat H_L(\gamma_L^{(-k)}(\sigma))\hat l(\sigma)\hat l^{(j-k)}(\sigma)\hat\eta^{(j-k)}(\sigma)\\
&-\sum_{L=0}^\infty\sum_{k=0}^L\sum_{j=0}^L\partial_\alpha^{(k)}\partial_\alpha^{(j)}\widehat H_L(\gamma_L^{(-k)}(\sigma))\hat l(\sigma)\hat l^{(j-k)}(\sigma)\hat\eta(\sigma)\\
=&-\hat l(\sigma)\mathscr E[\hat h](\theta),
\end{split}
\end{equation}
where $\hat l^{(j)}(\sigma)=\hat l(\sigma+j\omega\alpha)$ and $\hat \eta^{(j)}(\sigma)=\hat\eta(\sigma+j\omega\alpha)$.

For fixed $L\in\mathbb{N}$, we note that, when $j=k=0,\cdots,L$ the term in the first sum of the left-hand side of (\ref{y10}) cancels the one in the second sum. When $j\not=k$, we observe that we have four terms involving the mixed derivatives, that is
\begin{equation}\label{y11}
\begin{aligned}
\partial_\alpha^{(k)}\partial_\alpha^{(j)}&\widehat H _L(\gamma_L^{(-k)}(\sigma))\hat l(\sigma)\hat l^{(j-k)}(\sigma)\hat\eta^{(j-k)}(\sigma)\\
+&\partial_\alpha^{(j)}\partial_\alpha^{(k)}\widehat H _L(\gamma_L^{(-j)}(\sigma))\hat l(\sigma)\hat l^{(k-j)}(\sigma)\hat\eta^{(k-j)}(\sigma)\\
-&\partial_\alpha^{(k)}\partial_\alpha^{(j)}\widehat H _L(\gamma_L^{(-k)}(\sigma))\hat l^{(j-k)}(\sigma)\hat l(\sigma)\hat\eta(\sigma)\\
-&\partial_\alpha^{(j)}\partial_\alpha^{(k)}\widehat H _L(\gamma_L^{(-j)}(\sigma))\hat l^{(k-j)}(\sigma)\hat l(\sigma)\hat\eta(\sigma).
\end{aligned}
\end{equation}
We introduce the notations
\begin{equation}\label{y21}
\begin{split}
[\mathcal S_n\hat\eta](\sigma)&:=\hat\eta(\sigma+n\omega\alpha)-\hat\eta(\sigma)\quad\quad\forall n\in\mathbb Z,\hat\eta\in\mathscr A_\rho,\\
\mathcal C_{j,k,L}(\sigma)&:=\partial_\alpha^{(k)}\partial_\alpha^{(j)}\widehat H_L(\gamma_L^{(-k)}(\sigma))\hat l(\sigma)\hat l^{(j-k)}(\sigma).
\end{split}
\end{equation}
With notations above, we can rearrange (\ref{y11}) as
\begin{equation*}
\begin{split}
\mathcal C_{j,k,L}(\sigma)&\cdot[\hat\eta^{(j-k)}-\hat\eta](\sigma)\\
&-\mathcal C_{j,k,L}(\sigma+(k-j)\omega\alpha)[\hat\eta^{(j-k)}-\hat\eta](\sigma+(k-j)\omega\alpha)\\
&=-\mathcal S_{k-j}[\mathcal C_{j,k,L}\mathcal S_{j-k}\hat\eta](\sigma).
\end{split}
\end{equation*}
Therefore, (\ref{y10}) can be written as
\begin{equation}\label{y12}
    \sum_{L=1}^\infty\sum_{\substack{k,j=0\\k>j}}^L\mathcal S_{k-j}[\mathcal C_{j,k,L}\mathcal S_{j-k}\hat\eta](\sigma)=\hat l(\sigma)\mathscr E[\hat h](\theta).
\end{equation}

\subsubsection{Perturbative treatment}
We will study conditions for the invertibility of the operators $\mathcal S_n$. In fact, $\mathcal S _n:{\mathscr A}_\rho\to
\overset{\circ}{\mathscr A}_\rho$ is diagonal on Fourier series. Due to the Diophantine condition, for any given $\hat\eta\in\overset{\circ}{\mathscr A}_\rho$, ,we can find the solution of 
\begin{equation*}
\mathcal S _n\hat\gamma=\hat\eta.
\end{equation*}
These solutions $\hat\gamma$ are unique up to additive constants. We will denote by $\mathcal S ^{-1}_n$ the operator that given $\hat\eta$ produces the $\hat\gamma$ with zero average. 

Hence, we can define the operators
$$\mathcal L_n^{\pm}:=\mathcal S^{-1}_{\pm1}\mathcal S_n$$
acting on ${\mathscr A}_\rho$ and the operators
$$\mathcal R_n^{\pm}:=\mathcal S_n\mathcal S^{-1}_{\pm1}$$
acting on $\overset{\circ}{\mathscr A}_\rho$.
Therefore, we have the following lemma.
\begin{lem}[See \cite{ref6}]
$\|\mathcal L_n^\pm\|_{\mathcal M_1(\mathscr A_\rho,\overset{\circ}{\mathscr A_{\rho}})}\le|n|,\|\mathcal R_n^\pm\|_{\mathcal{M}_1(\overset{\circ}{\mathscr A_\rho},\overset{\circ}{\mathscr A_\rho})}\le|n|$.
\end{lem}
Therefore, (\ref{y12}) can be written as
\begin{equation}\label{y13}
\begin{split}
\hat l(\sigma)\mathscr E[\hat h](\sigma)&=\mathcal S_1[\mathcal C_{0,1,1}\mathcal S_{-1}\hat\eta](\sigma)+\sum_{L=2}^\infty\sum_{0\le j<k\le L}\mathcal S_{k-j}[\mathcal C_{j,k,L}\mathcal S_{j-k}\hat\eta](\sigma)\\
&=\mathcal S_1[\mathcal C_{0,1,1}+\sum_{L=2}^\infty\sum_{0\le j<k\le L}\mathcal S_1^{-1}\mathcal S_{k-j}\mathcal C_{j,k,L}\mathcal S_{j-k}\mathcal S_{-1}^{-1}]\mathcal S_{-1}\hat\eta(\sigma)\\
&=:\mathcal S_1[\mathcal C_{0,1,1}+\mathscr G]\mathcal S_{-1}\hat\eta(\sigma).
\end{split}
\end{equation}
We denote
\begin{equation*}
\mathscr G:=\sum_{L=2}^\infty\sum_{0\le j<k\le L}\mathcal S_1^{-1}\mathcal S_{k-j}\mathcal C_{j,k,L}\mathcal S_{j-k}\mathcal S_{-1}^{-1}=\sum_{L=2}^\infty\sum_{\substack{k,j=0\\k>j}}^L\mathcal L_{k-j}^+\mathcal C_{j,k,L}\mathcal R_{j-k}^-.
\end{equation*}

The procedure to solve (\ref{y13}) is similar to that of equation (\ref{9}) in the short-range case. We will give a brief description below.

The equation(\ref{y13}) is equivalent to the following system:
\begin{equation}\label{S1}
\mathcal S_1\widehat W(\sigma)=\hat l(\sigma)\mathscr E[\hat h](\theta)
\end{equation}
\begin{equation}\label{S2}
\mathcal S_{-1}[\eta](\sigma)=[\mathcal C_{0,1,1}+\mathscr G]^{-1}\widehat W(\sigma).
\end{equation}
It is easy to see that
\[
\begin{split}
\int_{\mathbb{T^N}}\hat l(\sigma)\mathscr E[\hat h]d\sigma&=\sum_{L\in\mathbb N}\sum_{k=0}^{L}\int_{\mathbb{T^N}}\hat l(\sigma)\partial_\alpha^{(k)}\widehat H_L(\gamma_L^{(-k)}(\sigma))d\sigma\\
&=\sum_{L\in\mathbb N}\sum_{k=0}^{L}\int_{\mathbb{T^N}}\partial_\alpha^{(k)}\widehat H_L(\gamma_L(\sigma))\hat l^{(k)}(\sigma)d\sigma\\
&=\sum_{L\in\mathbb N}\int_{\mathbb{T^N}}d\widehat H_L(\gamma_L(\sigma))=0.
\end{split}
\]
We write $\widehat W=\widehat W^0+\overline{\widehat W}$, where $\widehat W^0$ has zero average and $\overline{\widehat W}$ is some constant such that
\[
\int_{\mathbb T ^{\mathbb N}}(\mathcal C_{0,1,1}+\mathscr G)^{-1}[\widehat W]d\sigma=0.
\]
According to Lemma \ref{l1}, we can solve equation (\ref{S1}) and (\ref{S2}), and estimate the solutions. Thus, we obtain a better solution: $\hat h+\widehat\Delta=\hat h+\hat l\cdot\hat\eta$.


\subsubsection{Estimates for one iterative step}\label{5.5}
\paragraph{\textbf{Estimates for approximate solutions}}
Due to the assumptions (\rm{H2}) and (\rm{H3}) in Theorem \ref{lr}, we obtain the following estimate:
\begin{equation*}
\begin{split}
\|\mathscr G\|_{\mathcal{M}_1(\overset{\circ}{\mathscr A_\rho},\overset{\circ}{\mathscr A_\rho})}&\le\sum_{L=2}^\infty M_L(N^+)^2\sum_{\substack{j,k=0\\ k>j}}^L|j-k|^2\\
&\le C\sum_{L=2}^\infty M_LL^4=:\beta.
\end{split}
\end{equation*}
Hence, by (\rm{H5}) the usual Neumann series shows that the operator $(\mathcal C_{0,1,1}+\mathscr G)^{-1}$ is boundedly invertible from $\overset{\circ}{\mathscr A_\rho}$ to $\overset{\circ}{\mathscr A_\rho}$. Moreover, we have
\begin{equation}\label{y14}
\begin{split}
\|(\mathcal C_{0,1,1}+\mathscr G)^{-1}-\mathcal C_{0,1,1}^{-1}\|_{\mathcal{M}_1(\overset{\circ}{\mathscr A_\rho},\overset{\circ}{\mathscr A_\rho})}&=\|[\mathcal C_{0,1,1}(Id+\mathcal C^{-1}_{0,1,1}\mathscr G)]^{-1}-\mathcal C_{0,1,1}^{-1}\|_{\mathcal{M}_1(\overset{\circ}{\mathscr A_\rho},\overset{\circ}{\mathscr A_\rho})}\\
&=\|\sum_{j=0}^\infty(-\mathcal C^{-1}_{0,1,1}\mathscr G)^j\mathcal C_{0,1,1}^{-1}-\mathcal C_{0,1,1}^{-1}\|_{\mathcal{M}_1(\overset{\circ}{\mathscr A_\rho},\overset{\circ}{\mathscr A_\rho})}\\
&\le\|\mathcal C_{0,1,1}^{-1}\|_\rho\sum_{j=1}^\infty\|\mathcal C_{0,1,1}^{-1}\mathscr G\|^j_{\mathcal{M}_1(\overset{\circ}{\mathscr A_\rho},\overset{\circ}{\mathscr A_\rho})}\\
&\le(N^-)^2T\frac{(N^-)^2T\beta}{1-(N^-)^2T\beta}\\
&\le(N^-)^2T.
\end{split}
\end{equation}

The equation for $\overline{\widehat W}\in\mathbb{R}$ can be written as
$$\int_{\mathbb T^{\mathbb N}_\rho}\mathcal C_{0,1,1}^{-1}\overline{\widehat W}d\sigma+\int_{\mathbb T_\rho^{\mathbb N}}[(\mathcal C_{0,1,1}+\mathscr G)^{-1}-\mathcal C_{0,1,1}^{-1}][\widehat W]d\sigma=-\int_{\mathbb T_\rho^{\mathbb N}}\mathcal C_{0,1,1}^{-1}\widehat W^0d\sigma.$$
Therefore, we have
$$|\overline{\widehat W}|\le U\|(\mathcal C_{0,1,1}+\mathscr G)^{-1}-\mathcal C_{0,1,1}^{-1}\|_{\mathcal{M}_1(\overset{\circ}{\mathscr A_\rho},\overset{\circ}{\mathscr A_\rho})}(\|\widehat W^0\|_\rho+|\overline{\widehat W}|)+U(N^-)^2T\|\widehat W^0\|_\rho.$$

By assumption (\rm{H5}) of Theorem \ref{lr} and the estimates (\ref{y14}), we obtain $|\overline{\widehat W}|\le2\|\widehat W^0\|_\rho$. Hence
\begin{equation}\label{y15}
\|\widehat W\|_\rho\le3\|\widehat W^0\|_\rho.
\end{equation}

Since we will lose domain repeatedly, we will introduce auxiliary positive numbers $\rho'<\rho''<\rho$ such that $\rho''=\rho-\frac{\bar\sigma}{2}$, where we denote $\bar\sigma=\rho-\rho'$.
By assumption (\rm{H2}) in Theorem \ref{lr}, we have
$$\|\hat l\cdot\mathscr E[\hat h]\|_\rho\le N^+\|\mathscr E[\hat h]\|_\rho.$$
Then, by Lemma \ref{l1}, we have
$$\|\widehat W^0\|_{\rho''}\le C\nu^{-1}{\rm exp}\left(\frac{2^{\frac1s}\tilde\tau}{\bar\sigma^\frac1s}{\rm ln}\frac{2\tilde\tau}{\bar\sigma}\right)\|\hat l\cdot\mathscr E[\hat h]\|_\rho\le C\nu^{-1}{\rm exp}\left(\frac{2^{\frac1s}\tilde\tau}{\bar\sigma^\frac1s}{\rm ln}\frac{2\tilde\tau}{\bar\sigma}\right)N^+\|\mathscr E[\hat h]\|_\rho.$$
Therefore, due to (\ref{y14}), (\ref{y15}) and Lemma \ref{l1}, we obtain:
\begin{equation*}
\begin{split}
\|\hat\eta\|_{\rho'}&\le C\nu^{-1}{\rm exp}\left(\frac{2^{\frac1s}\tilde\tau}{\bar\sigma^\frac1s}{\rm ln}\frac{2\tilde\tau}{\bar\sigma}\right)\cdot2(N^-)^2T\cdot3\|\widehat W^0\|_{\rho''}\\
&\le6C^2\nu^{-2}{\rm exp}\left(\frac{2^{\frac1s+1}\tilde\tau}{\bar\sigma^\frac1s}{\rm ln}\frac{2\tilde\tau}{\bar\sigma}\right)N^+(N^-)^2T\|\mathscr E[\hat h]\|_{\rho}.
\end{split}
\end{equation*}
Hence, we will have the estimates for the solution $\hat\Delta$ of (\ref{y7})
\begin{equation}\label{y19a}
\|\widehat\Delta\|_{\rho'}\le\|\hat l\|_{\rho'}\|\hat\eta\|_{\rho'}\le6C^2\nu^{-2}{\rm exp}\left(\frac{2^{\frac1s+1}\tilde\tau}{\bar\sigma^\frac1s}{\rm ln}\frac{2\tilde\tau}{\bar\sigma}\right)(N^+)^2(N^-)^2T\|\mathscr E[\hat h]\|_{\rho}.
\end{equation}

In order to make sure $\mathscr E[\hat h+\widehat \Delta]$ is well-defined, by Lemma \ref{l7}, it is sufficient to let $\|\widehat\Delta\|_{\rho'}<\iota$. We will show that it can hold for sufficiently small $\|\mathscr E[\hat h]\|_{\rho_0}$ in the later.
Formally, we have
\[
\begin{split}
\mathscr E[\hat h+\widehat\Delta]&=(\mathscr E[\hat h+\widehat\Delta]-\mathscr E[\hat h]-D\mathscr E[\hat h]\widehat\Delta)+\hat l^{-1}(\hat l\mathscr E[\hat h]+\hat l(D\mathscr E[\hat h]\widehat\Delta))\\
&=(\mathscr E[\hat h+\widehat\Delta]-\mathscr E[\hat h]-D\mathscr E[\hat h]\widehat\Delta)+\hat l^{-1}\widehat\Delta(D\mathscr E[\hat h]\hat l)\\
&=(\mathscr E[\hat h+\widehat\Delta]-\mathscr E[\hat h]-D\mathscr E[\hat h]\widehat\Delta)+\hat l^{-1}\widehat\Delta\frac{d}{d\theta}\mathscr E[\hat h].
\end{split}
\]
The second equation uses $\widehat\Delta$ is the solution of (\ref{y7}) and the third identity is just (\ref{y16}).

By Lemma \ref{l4}, we have
$$\|\frac{d}{d\theta}\mathscr E[\hat h]\|_{\rho'}\le\bar\sigma^{-1}\|\mathscr E[\hat h]\|_\rho,$$
then 
\begin{equation}\label{y17}
\|\frac{\widehat\Delta}{\hat l}\cdot\frac{d}{d\theta}\mathscr E[\hat h]\|_{\rho'}\le6C^2\nu^{-2}{\rm exp}\left(\frac{2^{\frac1s+1}\tilde\tau}{\bar\sigma^\frac1s}{\rm ln}\frac{2\tilde\tau}{\bar\sigma}\right)(N^+)^2(N^-)^3T\bar\sigma^{-1}\|\mathscr E[\hat h]\|_{\rho}^2.
\end{equation}

We also obtain the inequality by Lemma \ref{l7}
\begin{equation}\label{y18}
\begin{split}
\quad\|\mathscr E[\hat h+\widehat\Delta]&-\mathscr E[\hat h]-D\mathscr E[\hat h]\widehat\Delta\|_{\rho'}\le 
\sum_{L=0}^{\infty}\sum_{k=0}^LM_L(L+1)^2\|\widehat\Delta\|_{\rho'}^2 =\sum_{L=0}^\infty M_L(L+1)^3\|\widehat\Delta\|^2_{\rho'}\\
&\le36\sum_{L=0}^\infty M_L(L+1)^3C^4\nu^{-4}{\rm exp}\left(\frac{2^{\frac1s+2}\tilde\tau}{\bar\sigma^\frac1s}{\rm ln}\frac{2\tilde\tau}{\bar\sigma}\right)(N^+)^4(N^-)^4T^2\|\mathscr E[\hat h]\|_\rho^2.
\end{split}
\end{equation}
By (\ref{y17}) and (\ref{y18}) we have
\begin{equation}\label{y20}
\begin{split}
\|\mathscr E[\hat h+\widehat\Delta]\|_{\rho'}&\le36C^2\nu^{-2}{\rm exp}\left(\frac{2^{\frac1s+2}\tilde\tau}{\bar\sigma^\frac1s}{\rm ln}\frac{2\tilde\tau}{\bar\sigma}\right)(N^+)^2(N^-)^3T\\
&\qquad\qquad\cdot(\bar\sigma^{-1}+\sum_{L=0}^\infty M_L(L+1)^3C^2\nu^{-2}(N^+)^2N^-T)\|\mathscr E[\hat h]\|^2_\rho\\
&\le A{\rm exp}\left(\frac{2^{\frac1s+2}\tilde\tau}{\bar\sigma^\frac1s}{\rm ln}\frac{2\tilde\tau}{\bar\sigma}\right)(\bar\sigma^{-1}+B)\|\mathscr E[\hat h]\|^2_\rho.
\end{split}
\end{equation}
where we denote by $A$, $B$ the uniform upper bounds of $36C^2\nu^{-2}(N^+)^2(N^-)^3T$,\\$\sum_{L=0}^\infty M_L(L+1)^3C^2\nu^{-2}(N^+)^2N^-T$  respectively, and we require $\bar\sigma\le2\tilde\tau$.
\smallskip

\paragraph{\textbf{Estimates for the condition numbers}}
We use the notations introduced in (\ref{y2}) and denote by $\tilde\gamma_L$ the one corresponding to $\hat h+\widehat\Delta$ instead of $\hat h$. By Cauchy estimate and the mean value theorem, we have
\begin{equation*}
\begin{split}
&\|\partial_\alpha^{(0)}\partial_\alpha^{(1)}\widehat H_1(\tilde\gamma_1(\sigma))-\partial_\alpha^{(0)}\partial_\alpha^{(1)}\widehat H_1(\gamma_1(\sigma))\|_{\rho'}\le2M_1\|\widehat\Delta\|_{\rho'}\\
&\le2M_1\cdot6C^2\nu^{-2}{\rm exp}\left(\frac{2^{\frac1s+1}\tilde\tau}{\bar\sigma^\frac1s}{\rm ln}\frac{2\tilde\tau}{\bar\sigma}\right)(N^+)^2(N^-)^2T\|\mathscr E[\hat h]\|_{\rho}.
\end{split}
\end{equation*}

We define
\begin{equation*}
\begin{split}
&\chi:=6C^2\nu^{-2}{\rm exp}\left(\frac{2^{\frac1s+1}\tilde\tau}{\bar\sigma^\frac1s}{\rm ln}\frac{2\tilde\tau}{\bar\sigma}\right)(N^+)^2(N^-)^2T\|\mathscr E[\hat h]\|_{\rho},\\
&\chi':=6C^2\nu^{-2}\cdot\frac{2}{\bar\sigma}{\rm exp}\left(\frac{2^{\frac2s+1}\tilde\tau}{\bar\sigma^\frac1s}{\rm ln}\frac{4\tilde\tau}{\bar\sigma}\right)(N^+)^2(N^-)^2T\|\mathscr E[\hat h]\|_{\rho}.
\end{split}
\end{equation*}

Replacing $\bar\sigma$ with $\frac{\bar\sigma}2$ in (\ref{y19a}), we get
\[
\begin{split}
&\|\widehat\Delta\|_{\rho''}\le6C^2\nu^{-2}{\rm exp}\left(\frac{2^{\frac2s+1}\tilde\tau}{\bar\sigma^\frac1s}{\rm ln}\frac{4\tilde\tau}{\bar\sigma}\right)(N^+)^2(N^-)^2T\|\mathscr E[\hat h]\|_{\rho},\\
&\|\partial_\alpha\widehat\Delta\|_{\rho'}\le(\rho''-\rho')^{-1}\|\widehat\Delta\|_{\rho''}\le\chi'.
\end{split}
\]
If $\chi'\le(N^+)^2+N^+$ and $\frac{\bar\sigma}2\le1$, then $\chi\le\chi'$ and we obtain:
\begin{equation*}
\begin{split}
\|\widetilde{\mathcal C}_{0,1,1}-\mathcal C_{0,1,1}\|_{\rho'}&=\|\partial_\alpha^{(0)}\partial_\alpha^{(1)}\widehat H_1(\tilde\gamma_1^{(-1)}(\sigma))(\hat l+\partial_\alpha\widehat\Delta)(\sigma)(\hat l+\partial_\alpha\widehat\Delta)(\sigma-\omega\alpha)\\
&-\partial_\alpha^{(0)}\partial_\alpha^{(1)}\widehat H_1(\gamma_1^{(-1)}(\sigma))\hat l(\sigma)\hat l(\sigma-\omega\alpha)\|_{\rho'}\\
&\le\|\partial_\alpha^{(0)}\partial_\alpha^{(1)}\widehat H_1(\tilde\gamma_1^{(-1)}(\sigma))-\partial_\alpha^{(0)}\partial_\alpha^{(1)}\widehat H_1(\gamma_1^{(-1)}(\sigma))\|_{\rho'}\|\hat l(\sigma)\hat l(\sigma-\omega\alpha)\|_{\rho'}\\
&+\|\partial_\alpha^{(0)}\partial_\alpha^{(1)}\widehat H_1(\tilde\gamma_1^{(-1)}(\sigma))\hat l(\sigma)\partial_\alpha\widehat\Delta(\sigma-\omega\alpha)\|_{\rho'}\\
&+\|\partial_\alpha^{(0)}\partial_\alpha^{(1)}\widehat H_1(\tilde\gamma_1^{(-1)}(\sigma))\partial_\alpha\widehat\Delta(\sigma)\hat l(\sigma-\omega\alpha)\|_{\rho'}\\
&+\|\partial_\alpha^{(0)}\partial_\alpha^{(1)}\widehat H_1(\tilde\gamma_1^{(-1)}(\sigma))\partial_\alpha\widehat\Delta(\sigma)\partial_\alpha\widehat\Delta(\sigma-\omega\alpha)\|_{\rho'}\\
&\le2M_1\chi(N^+)^2+2M_1N^+\chi'+M_1(\chi')^2\\
&\le3M_1\chi'((N^+)^2+N^+).
\end{split}
\end{equation*}

We use the same notations as in Theorem \ref{lr}, but use the $\sim$ to indicate that they are estimated at the function $\hat h+\widehat\Delta$. Therefore, it is easy to check by the mean value theorem and Cauchy estimates:
\begin{equation}\label{y19}
\begin{split}
\widetilde N^-&\equiv\|(1+\partial_\alpha(\hat h+\widehat\Delta))^{-1}\|_{\rho'}\le N^-+\sum_{j=1}^{\infty}\|(-\frac{\partial_\alpha\widehat\Delta}{1+\partial_\alpha\hat h})^j(1+\partial_\alpha\hat h)^{-1}\|_{\rho'}\\
&\le N^-+\sum_{j=1}^{\infty}(N^-\chi')^jN^-=N^-+\frac{\chi'(N^-)^2}{1-\chi'N^-},\\
\widetilde U^{-1}&\equiv\left |\int_{\mathbb T^{\mathbb N}}\widetilde{\mathcal C}_{0,1,1}^{-1}\right |=\left |\int_{\mathbb T^{\mathbb N}}\left \{\mathcal C_{0,1,1}[Id+\mathcal C_{0,1,1}^{-1}(\widetilde{\mathcal C}_{0,1,1}-\mathcal C_{0,1,1})]\right \}^{-1}\right |\\
&\ge U^{-1}(1-\sum_{j=1}^{\infty}\|\mathcal C_{0,1,1}^{-1}(\widetilde{\mathcal C}_{0,1,1}-\mathcal C_{0,1,1})\|_{\rho'}^j)\\
&\ge U^{-1}-\frac{3(N^-)^2TM_1((N^+)^2+N^+)\chi'}{1-3(N^-)^2TM_1((N^+)^2+N^+)\chi'}U^{-1},\\
|\tilde c-c|&\equiv\left |\langle\frac1{(\hat l+\widehat{\Delta})\cdot(\hat l+\widehat{\Delta})\circ T_{-\omega\alpha}}\rangle-\langle\frac1{\hat l\cdot\hat l\circ T_{-\omega\alpha}}\rangle\right |\\
&=\left |\langle\frac{(\hat l\cdot\widehat\Delta\circ T_{-\omega\alpha}+\widehat\Delta\cdot(\hat l+\widehat\Delta)\circ T_{-\omega\alpha})}{(\hat l+\widehat{\Delta})\cdot(\hat l+\widehat{\Delta})\circ T_{-\omega\alpha}\cdot\hat l\cdot\hat l\circ T_{-\omega\alpha}}\rangle\right |\\
&\le(\widetilde N^-)^2(N^-)^2\chi(2N^++\chi).
\end{split}
\end{equation}

\subsubsection{Estimates for the iterative steps}
The function $\hat h_0$ we start with is defined in a domain parameterized by $\rho_0$. We choose a sequence of parameters
$$\rho_n=\rho_{n-1}-\frac{\rho_0}{4}2^{-n}.$$
The $n$ iterative step starts with a function $\hat h_n$ defined in a domain of radius $\rho_n$ and ends up with a function $\hat h_{n+1}=\hat h_n+\widehat\Delta_n$ defined in a domain of radius $\rho_{n+1}$.

Due to the non-degeneracy condition, (\ref{y19}) are bounded uniformly. By (\ref{y20}), if we have $\rho_0\le\frac4B$ then 
\begin{equation*}
\begin{split}
\epsilon_n&\le A{\rm exp}\left(\frac{2^{2+\frac{n+3}{s}}\tilde\tau}{{\rho_0}^\frac1s}{\rm ln}\frac{2^{n+3}\tilde\tau}{\rho_0}\right)\cdot\left(\frac{2^{n+2}}{\rho_0}+B\right)\epsilon_{n-1}^2\\
&\le A{\rm exp}\left(\frac{2^{2+\frac{n+3}{s}}\tilde\tau}{{\rho_0}^\frac1s}{\rm ln}\frac{2^{n+3}\tilde\tau}{\rho_0}\right)\frac{2^{n+3}}{\rho_0}\epsilon_{n-1}^2\\
&\le(D\epsilon_0)^{2^n},
\end{split}
\end{equation*}
where we define
$D:=\frac{A}{\rho_0}2^82^{\frac{2^{8+\frac1s}\tilde\tau}{\rho_0^{\frac1s}}}\left(\frac{\tilde\tau}{\rho_0}\right)^{\frac{2^{5+\frac1s}\tilde\tau}{\rho_0^\frac1s}}$. If $D\epsilon_0< 1$, then $\epsilon_n$ decreases faster than any exponential.

We denote by $E$ the uniform upper bound of $\chi$, 
\begin{equation*}
\begin{split}
\sum_{j=0}^n\|\widehat\Delta_j\|_{\rho_n}&\le\sum_{j=0}^n\|\widehat\Delta_j\|_{\rho_j}\le\sum_{j=0}^nE{\rm exp}\left(\frac{2^{1+\frac{j+3}{s}}\tilde\tau}{\rho_0^\frac1s}{\rm ln}\frac{2^{j+3}\tilde\tau}{\rho_0}\right)\epsilon_j\\
&\le\sum_{j=0}^nE{\rm exp}\left(\frac{2^{1+\frac{j+3}{s}}\tilde\tau}{\rho_0^\frac1s}{\rm ln}\frac{2^{j+3}\tilde\tau}{\rho_0}\right)(D\epsilon_0)^{2^j}\\
&\le\sum_{j=0}^nE(D'\epsilon_0)^{2^j},
\end{split}
\end{equation*}
where we define
$D'=e^{2\tilde\tau^{\frac32}\frac8{\rho_0}}D.$
If $\epsilon_0<\min\{\frac{\iota}{8ED'}, \frac{1}{2D'} \}$, then we have $\sum_{j=0}^n\|\widehat\Delta_j\|_{\rho_n}<\frac{\iota}{4},\,\forall n\in\mathbb N$, and $\mathscr E[\hat h_{n+1}]$ is well-defined.

We will verify that condition numbers in assumption (\rm{H2}) of Theorem \ref{lr} have uniform upper bounds respectively. Take $N^+$ as an example below: 
\begin{equation*}
\begin{split}
N^+(\hat h_n,\rho_n)&\le N^+(\hat h_{n-1},\rho_n)+\|\partial_\alpha(\hat h_n-\hat h_{n-1})\|_{\rho_n}\\
&\le N^+(\hat h_{n-1},\rho_{n-1})+(\rho_{n-1}-\rho_n)^{-1}\|\widehat\Delta_{n-1}\|_{\rho_{n-1}}\\
&\le N^+(\hat h_0,\rho_0)+\frac{8}{\rho_0}\sum_{j=0}^{n-1}2^jL(D'\epsilon_0)^{2^j}\\
&\le N^+(\hat h_0,\rho_0)+\frac{8}{\rho_0}\sum_{j=0}^{n-1}L(2D'\epsilon_0)^{2^j},
\end{split}
\end{equation*}
Hence, if $\epsilon_0\le\min\{\frac{\rho_0 N^+(\hat h_0,\rho_0)}{32LD'},\frac{1}{4D'}\}$, we have $N^+(\hat h_n,\rho_n)\le2 N^+(\hat h_0,\rho_0)$, $\forall n\in\mathbb N$. Therefore, condition numbers are uniformly bounded and the iterative step
can be carried out infinitely often.

In order to finish the proof of Theorem \ref{lr}, here we give estimates of the solution:
\begin{equation*}
\begin{split}
\|\hat h_N-\hat h_0\|_{\frac{\rho_0}{2}}&\le\sum_{n=1}^N \|\widehat\Delta_n\|_{\rho_n}\\
&\le\sum_{n=1}^{\infty} L(D'\epsilon_0)^{2^n}=LD'\epsilon_0\frac1{1-D'\epsilon_0}\\
&\le2LD'\epsilon_0,
\end{split}
\end{equation*}
where $\epsilon_0\le\frac1{2D'}$. Therefore $\|\hat h^*-\hat h_0\|_{\frac{\rho_0}{2}}\le C_1\epsilon_0$, where $C_1=2LD'$. 

\subsubsection{Uniqueness of the solution}
Suppose that $\| \hat h^*-\hat h^{**} \|_{\frac{3}{8}\rho_0},\,\| \hat h^*-\hat h^{**} \|_{\frac{3}{8}\rho_0}\le r <\frac{\iota}{4}$, $\mathcal{E}[\hat h^*]=\mathcal E[\hat h^{**}]=0$, and $\langle \hat h^* \rangle=\langle \hat h^{**} \rangle=0$. Then we have $\|\hat h^{**}-\hat h^*\|_{\frac{\rho_0}{4}}\le\|\hat h^{**}-\hat h^*\|_{\frac{3}{8}\rho_0}\le 2r$, and
\begin{equation}\label{uu1}
    \begin{aligned}
        0=\mathcal E[\hat h^{**}]-\mathcal E[\hat h^{*}]=&D\mathcal{E}[\hat h^*](\hat h^{**}-\hat h^*)+\sum_{L=0}^{\infty}\sum_{k=0}^L(\partial_{\alpha}^{(k)}\widehat H_L(\gamma_L^{**(-k)}(\theta))-\partial_{\alpha}^{(k)}\widehat H_L(\gamma_L^{*(-k)}(\theta))\\
        &\quad-\sum_{j=0}^L\partial_{\alpha}^{(k)}\widehat H_L(\gamma_L^{*(-k)}(\theta))(\hat h^{**}-\hat h^*)(\sigma-k\omega\alpha+j\omega\alpha))\\
        =&D\mathcal{E}[\hat h^*](\hat h^{**}-\hat h^*)+R.
    \end{aligned}
\end{equation}
By Lemma \ref{l0} and Lemma \ref{l7}, we have 
\begin{align*}
    \|R\|_{\frac{\rho_0}{4}}&\le\sum_{L=0}^{\infty}\sum_{k=0}^L M_L(L+1)^2\| \hat h^{**}-\hat h^*\|_{\frac{\rho_0}{4}}^2=\sum_{L=0}^{\infty}M_L(L+1)^3\| \hat h^{**}-\hat h^*\|_{\frac{\rho_0}{4}}^2\\
    &\le \sum_{L=0}^{\infty}M_L(L+1)^3\| \hat h^{**}-\hat h^*\|_{\frac{\rho_0}{8}}\| \hat h^{**}-\hat h^*\|_{\frac{3}{8}\rho_0}\\
    &\le \sum_{L=0}^{\infty}M_L(L+1)^3\| \hat h^{**}-\hat h^*\|_{\frac{\rho_0}{8}} \cdot 2r.
\end{align*}

Denote $\hat l^*=1+\partial_{\alpha}\hat h^*$. Since $$D\mathcal{E}[\hat h^*]\cdot \hat l^*=\frac{d}{d\theta}\mathcal{E}[\hat h^*](\theta)=0,$$
we can write the equation (\ref{uu1}) as:
$$\hat l^* D\mathcal{E}[\hat h^*](\hat h^{**}-\hat h^*)-(\hat h^{**}-\hat h^*)D\mathcal E[\hat h^*]\cdot\hat l^*=-\hat l^*R.$$

Notice that the equation (\ref{uu1}) and the equation (\ref{y7}) share the same form. Using the uniqueness statements for the solution of equation (\ref{y7}) and the estimate (\ref{y19a}), we conclude that 
\begin{equation*}
\begin{aligned}
    \|\hat h^{**}-\hat h^*\|_{\frac{\rho_0}{8}}&\le6C^2\nu^{-2}{\rm exp}\left(\frac{2^{\frac1s+1}\tilde\tau}{(\frac{\rho_0}{8})^\frac1s}{\rm ln}\frac{2\tilde\tau}{\frac{\rho_0}{8}}\right)(N^+)^2(N^-)^2T\|R\|_{\frac{\rho_0}{4}}\\
    &\le 6C^2\nu^{-2}{\rm exp}\left(\frac{2^{\frac1s+1}\tilde\tau}{(\frac{\rho_0}{8})^\frac1s}{\rm ln}\frac{2\tilde\tau}{\frac{\rho_0}{8}}\right)(N^+)^2(N^-)^2T\sum_{L=0}^{\infty}M_L(L+1)^3\| \hat h^{**}-\hat h^*\|_{\frac{\rho_0}{8}} \cdot 2r
\end{aligned}
\end{equation*}
Therefore, when $r$ is small enough, we obtain $\hat h^{**}=\hat h^*$. This completes the proof of uniqueness of the solution in Theorem \ref{lr}.
\bibliographystyle{alpha}
\bibliography{a}
\end{document}